\documentclass{amsart}

\usepackage{amsmath,amssymb, hyperref, color}

\newtheorem{theorem}{Theorem}[section]
\newtheorem{lemma}[theorem]{Lemma}

\newtheorem{proposition}[theorem]{Proposition}
\newtheorem{conjecture}[theorem]{Conjecture}

\theoremstyle{definition}

\newtheorem{definition}[theorem]{Definition}

\newcommand{\N}{\mathbb N}
\newcommand{\Z}{\mathbb Z}
\newcommand{\R}{\mathbb R}
\newcommand{\Q}{\mathbb Q}

 \DeclareMathOperator{\ord}{ord}

 \DeclareMathOperator{\supp}{supp}

\renewcommand{\t}{\, | \,}

\newcommand{\be}{\begin{equation}}
\newcommand{\ee}{\end{equation}}

\newcommand{\bnml}{\begin{multline}}
\newcommand{\enml}{\end{multline}}
\newcommand{\buml}{\begin{multline*}}
\newcommand{\euml}{\end{multline*}}

\newcommand{\ber}{\begin{eqnarray}}
\newcommand{\eer}{\end{eqnarray}}

\numberwithin{equation}{section}

\begin{document}

\title[A characterization of finite abelian groups via sets of lengths]{A characterization of finite abelian groups via sets of lengths in transfer Krull monoids}

\author{Qinghai Zhong}
\address{Qinghai Zhong\\
University of Graz, NAWI Graz \\
Institute for Mathematics and Scientific Computing \\
Heinrichstra{\ss}e 36\\
8010 Graz, Austria}
\email{qinghai.zhong@uni-graz.at}
\urladdr{http://qinghai-zhong.weebly.com}



\thanks{This work was supported by
the Austrian Science Fund FWF, Project Number P 28864-N35.}
\begin{abstract}
Let $H$ be a transfer Krull monoid over a finite ablian group $G$ (for example, rings of integers, holomorphy rings in algebraic function fields, and regular congruence monoids in these domains). Then each nonunit $a \in H$ can be written as a product of irreducible elements, say $a = u_1 \ldots u_k$, and the number of factors $k$ is called the length of the factorization. The set $\mathsf L (a)$ of all possible factorization lengths is the set of lengths of $a$. It is classical that the system $\mathcal L (H) = \{ \mathsf L (a) \mid a \in H \}$ of all sets of lengths depends only on the group $G$, and a standing conjecture states that conversely the system $\mathcal L (H)$ is characteristic for the group $G$. Let $H'$ be a further transfer Krull monoid over a finite ablian group $G'$ and suppose that $\mathcal L (H)= \mathcal L (H')$. We prove that, if $G\cong C_n^r$ with $r\le n-3$ or ($r\ge n-1\ge 2$ and $n$ is a prime power),  then $G$ and $G'$ are isomorphic.
\end{abstract}

\maketitle

\medskip
\section{Introduction and Main Result} \label{1}
\medskip

Let $H$ be an atomic unit-cancellative monoid. Then each non-unit $a \in H$ can be written as a product of atoms, and if $a = u_1 \ldots u_k$ with atoms $u_1, \ldots, u_k$ of $H$, then $k$ is called the length of the factorization. The set $\mathsf L (a)$ of all possible factorization lengths is the set of lengths of $a$, and $\mathcal L (H) = \{ \mathsf L (a) \mid a \in H \}$ is called the system of sets of lengths of $H$ (for convenience we set $\mathsf L (a) = \{0\}$ if $a$ is an invertible element of $H$).  Under a variety of noetherian conditions on $H$ (e.g., $H$ is the monoid of nonzero elements of a commutative noetherian domain) all sets of lengths are finite. Sets of lengths (together with invariants controlling their structure, such as elasticities and sets of distances) are a well-studied means for describing the arithmetic structure of monoids.

Let $H$ be a transfer Krull monoid over a finite abelian group $G$. Then, by definition, there is a weak transfer homomorphism $\theta \colon H \to \mathcal B (G)$, where $\mathcal B (G)$ denotes the monoid of zero-sum sequences over $G$, and hence $\mathcal L (H) = \mathcal L \big( \mathcal B (G) \big)$. We use the abbreviation $\mathcal L (G) = \mathcal L \big( \mathcal B (G) \big)$.
 By a result due to  Carlitz in 1960, we know that $H$ is half-factorial (i.e., $|L|=1$ for all $L \in \mathcal L (H)$) if and only if $|G| \le 2$.
Suppose that $|G| \ge 3$. Then there is some $a \in H$ with $|\mathsf L (a)|>1$. If $k, \ell \in \mathsf L (a)$ with $k < \ell$ and $m \in \N$, then $\mathsf L (a^m) \supset \{km + \nu (\ell-k) \mid \nu \in [0,m] \}$ which shows that sets of lengths can become arbitrarily large.
Note that the system of sets of lengths of $H$ depends only on the class group $G$.
 The associated inverse question asks whether or not sets of lengths are characteristic for the class group. In fact, we have the following conjecture (it was first stated in \cite{Ge16c} and for  a detailed description of the background of
 this problem,  see  \cite{Ge16c}, \cite[Section 7.3]{Ge-HK06a}, \cite[page 42]{Ge-Ru09}, and \cite{Sc09b}).

\medskip
\begin{conjecture}
Let $G$ be a finite abelian group with $\mathsf D (G) \ge 4$. If $G'$ is an abelian group with $\mathcal L (G) = \mathcal L (G')$, then $G$ and $G'$ are isomorphic.
\end{conjecture}

\smallskip
Note if $\mathsf D(G)=3$, then we have $\mathcal L(C_3)=\mathcal L(C_2\oplus C_2)$. The system of sets of lengths $\mathcal L (G)$ is studied with methods from Additive Combinatorics. In particular, zero-sum theoretical invariants (such as the Davenport constant or the cross number) and the associated inverse problems play a crucial role.
Most of these invariants are well-understood only in a very limited number of cases (e.g., for groups of rank two, the precise value of the Davenport constant $\mathsf D (G)$ is known and the associated inverse problem is solved; however, if $n$ is not a prime power and $r \ge 3$, then the value of the Davenport constant $\mathsf D (C_n^r)$ is unknown). Thus it is not surprising that most affirmative answers to the Characterization Problem so far have been restricted to those groups where we have a good understanding of the Davenport constant. These groups include elementary $2$-groups, cyclic groups, and groups of rank two (for recent progress we refer to  \cite{Ba-Ge-Gr-Ph13a, Ge-Sc16a}).

The first and so far only groups, for which the Characterization Problem was solved whereas the Davenport constant is unknown, are groups of the form $C_n^r$, where $r, n\in \N$ and $2r<n-2$(this is done by \cite{Ge-Zh16b, Zh17a}), which use a deep  characterization of the structure of $\Delta^*(G)$. In this paper, we go on to study groups of the form $C_n^r$ and obtain the following theorem.

\smallskip
\begin{theorem} \label{main}
Let $H$ be a transfer Krull monoid over a finite abelian group $G$ with $\mathsf D(G)\ge 4$. Suppose $G\cong C_n^r$ with $r, n\in \N$ and $\mathcal L(H)=\mathcal L(H')$, where $H'$ is a further transfer Krull monoid over a finite abelian group $G'$. Then

\begin{enumerate}
\item If $r\le n-3$, then $G\cong G'$.

\item If $r\ge n-1$ and $n$ is a prime power, then $G\cong G'$.
\end{enumerate}
\end{theorem}

 This is made possible by introducing  new invariants $\rho(G, d)$ and $\rho^*(G,d)$ which are only depending on $\mathcal L(G)$(see Definitions \ref{3.1} and \ref{3.3}). In Section \ref{2} we gather the required background  both  on transfer Krull monoids as well as on Additive Combinatorics.  In Section \ref{3}, we provide a detailed study of $\rho(G,d)$ and $\rho^*(G,d)$.
  The proof of Theorem \ref{main}  will be provided in Section \ref{4}. The final section is concluding remarks and conjectures.

Throughout the paper, let $G\cong C_{n_1}\oplus \ldots \oplus C_{n_r}$ be a finite abelian group with $\mathsf D(G)\ge 4$, where $r, n_1,\ldots, n_r\in \N$ and $1<n_1\t \ldots \t n_r$.

\medskip
\section{Background on Transfer Krull monoids and  sets of lengths} \label{2}
\medskip

Our notation and terminology are consistent with \cite{Ge-HK06a, Ge-Sc16a}. For convenience, we set $\min \emptyset =0$. Let  $\mathbb N$ be  the set of positive integers, let $\mathbb Z$ be the set of integers, let $\Q$ be the set of rational numbers, and let $\R$ be the set of real numbers. For rational numbers $a, b \in \mathbb Q$, we denote by
 $[a, b ] = \{ x \in \mathbb Z \mid a \le x \le b\}$  the discrete, finite interval between $a$ and $b$. If $L \subset \mathbb N$ is a subset, then  $\Delta (L)$ denotes the {set of $($successive$)$ distances} of $L$ (that is, $d \in \Delta (L)$ if and only if $d = b-a$ with $a, b \in L$ distinct and $[a, b] \cap L = \{a, b\}$) and
  $\rho (L) = \sup L/ \min L$ denotes its elasticity (for convenience, we set $\rho ( \{0\})=1$).

 Let $r \in \N$ and let  $(e_1, \ldots, e_r)$ be an $r$-tuple of elements of $G$. Then $(e_1, \ldots, e_r)$ is said to be independent if $e_i \ne 0$ for all $i \in [1,r]$ and if for all $(m_1, \ldots, m_r) \in \Z^r$ an equation $m_1e_1+ \ldots + m_re_r=0$ implies that $m_ie_i=0$ for all $i \in [1,r]$.   Furthermore, $(e_1, \ldots, e_r)$ is said to be a basis of $G$ if it is independent and $G = \langle e_1 \rangle \oplus \ldots \oplus \langle e_r \rangle$. For every $n \in \N$, we denote by $C_n$ an additive cyclic group of order $n$. Since $G \cong C_{n_1} \oplus \ldots \oplus C_{n_r}$,  $r = \mathsf r (G)$ is the rank of $G$ and $n_r=\exp(G)$ is the exponent of $G$. 

\medskip
\noindent
{\bf Sets of Lengths.}
By a {monoid}, we mean an associative semigroup with unit element, and if not stated otherwise we use multiplicative notation. Let $H$ be a monoid with unit element $1=1_H \in H$.  An element $a \in H$ is said to be invertible (or an unit) if there exists an element $a'\in H$ such that $aa'=a'a=1$. The set of invertible elements of $H$ will be denoted by $H^{\times}$, and we say that $H$ is reduced if $H^{\times}=\{1\}$.  The monoid $H$ is said to be unit-cancellative if for any two elements $a,u \in H$,   each of the  equations $au=a$ or $ua=a$ implies that $u \in H^{\times}$. Clearly, every cancellative monoid is unit-cancellative.

Suppose that $H$ is unit-cancellative. An element $u \in H$ is said to be irreducible (or an atom) if $u \notin H^{\times}$ and for any two elements $a, b \in H$, $u=ab$ implies that $a \in H^{\times}$ or $b \in H^{\times}$. Let  $\mathcal A (H)$ denote the set of atoms, and we say that $H$ is atomic if every non-unit is a finite product of atoms. If $H$ satisfies the ascending chain condition on principal left ideals and on principal right ideals, then $H$ is atomic  (\cite[Theorem 2.6]{Fa-Tr18a}). If $a \in H \setminus H^{\times}$ and $a=u_1 \ldots u_k$, where $k \in \N$ and $u_1, \ldots, u_k \in \mathcal A (H)$, then $k$ is a factorization length of $a$, and
\[
\mathsf L_H (a) = \mathsf L (a) = \{k \mid k \ \text{is a factorization length of} \ a \} \subset \N
\]
denotes the {set of lengths} of $a$. It is convenient to set $\mathsf L (a) = \{0\}$ for all $a \in H^{\times}$. The family
\[
\mathcal L (H) = \{\mathsf L (a) \mid a \in H \}
\]
is called the {system of sets of lengths} of $H$, and
\[
\rho (H) = \sup \{\rho (L) \mid L \in \mathcal L (H) \} \in \R_{\ge 1} \cup \{\infty\}
\]
denotes the {elasticity} of $H$. We call
\[
\Delta (H) = \bigcup_{L \in \mathcal L (H)} \Delta (L) \ \subset \N
\]
 the set of distances of $H$. Note that $\Delta (H)$ can be infinite, and by definition we have $\Delta (H)= \emptyset$ if and only if $\rho (H)=1$. If $\Delta (H)\ne \emptyset$, then we have $\min \Delta (H) = \gcd \Delta (H)$(\cite[Proposition 2.9]{Fa-Ge-Ka-Tr17a}).

\medskip
\noindent
{\bf Monoids of zero-sum sequences.}

 Let $G_0 \subset G$ be a non-empty subset. Then $\langle G_0 \rangle$ denotes the subgroup generated by $G_0$. In Additive Combinatorics, a { sequence} (over $G_0$) means a finite sequence of terms from $G_0$ where repetition is allowed and the order of the elements is disregarded, and (as usual) we consider sequences as elements of the free abelian monoid with basis $G_0$. Let
\[
S = g_1 \ldots g_{\ell} = \prod_{g \in G_0} g^{\mathsf v_g (S)} \in \mathcal F (G_0)
\]
be a sequence over $G_0$. We call
\[
\begin{aligned}
\supp (S)&  = \{g \in G \mid \mathsf v_g (S) > 0 \} \subset G \ \text{the \ {\it support} \ of \ $S$} \,,\\
|S|  &= \ell = \sum_{g \in G} \mathsf v_g (S) \in \mathbb N_0 \
\text{the \ {\it length} \ of \ $S$} \,,   \\
\sigma (S) & = \sum_{i = 1}^{\ell} g_i \ \text{the \ {\it sum} \ of \
$S$} \,, \\
\ \ \text{ and }\ \ \  \Sigma (S) &= \Big\{ \sum_{i \in I} g_i
\mid \emptyset \ne I \subset [1,\ell] \Big\} \ \text{ the \ {\it set of
subsequence sums} \ of \ $S$} \,.
\end{aligned}
\]
The sequence $S$ is said to be
\begin{itemize}
\item {\it zero-sum free} \ if \ $0 \notin \Sigma (S)$,

\item a {\it zero-sum sequence} \ if \ $\sigma (S) = 0$,

\item a {\it minimal zero-sum sequence} \ if it is a nontrivial zero-sum
      sequence and every proper  subsequence is zero-sum free.
\end{itemize}
The set of zero-sum sequences $\mathcal B (G_0) = \{S \in \mathcal F (G_0) \mid \sigma (S)=0\} \subset \mathcal F (G_0)$ is a submonoid, and the set of minimal zero-sum sequences is the set of atoms of $\mathcal B (G_0)$.
For any arithmetical invariant $*(H)$ defined for a monoid $H$, we write $*(G_0)$ instead of $*(\mathcal B (G_0))$. In particular, $\mathcal A (G_0) = \mathcal A (\mathcal B (G_0))$ is the set of atoms of $\mathcal B (G_0)$, $\mathcal L (G_0)=\mathcal L (B(G_0))$ is the system of sets of lengths of $\mathcal B (G_0)$, and so on.   We denote by
\[
\mathsf D (G_0) = \max \{ |S| \mid S \in \mathcal A (G_0) \} \in \N 
\]
the { Davenport constant} of $G_0$. 

\medskip\noindent{\bf Transfer Krull monoids}
Let $H$ be a atomic unit-cancellative monoid. 
\begin{enumerate}
\item[(i)]We say a monoid homomorphism $\theta \colon H \to B$ to an atomic unit-cancellative monoid $B$ is  a {weak transfer homomorphism} if it has the following two properties:

\begin{itemize}
\item $B = B^{\times} \theta (H) B^{\times}$ and $\theta^{-1} (B^{\times})=H^{\times}$.
\item If $a \in H$, $n \in \N$, $v_1, \ldots, v_n \in \mathcal A (B)$ and $\theta (a) = v_1 \ldots v_n$, then there exist $u_1, \ldots, u_n \in \mathcal A (H)$ and a permutation $\tau \in \mathfrak S_n$ such that $a = u_1 \ldots u_n$ and $\theta (u_i) \in B^{\times} v_{\tau (i)} B^{\times}$ for each $i \in [1,n]$.
\end{itemize}
Let $\theta \colon H \to B$ be a weak transfer homomorphism between atomic unit-cancellative monoids. It follows that  for all $a \in H$,  we have $\mathsf L_H (a) = \mathsf L_B ( \theta (a))$ and hence $\mathcal L (H) = \mathcal L (B)$.

\item[(ii)]   We say
  $H$ is  a  {\it transfer Krull monoid}  if one of the following equivalent conditions holds:
  \begin{itemize}

  \item There exists  a weak transfer homomorphism $\theta \colon H \to \mathcal H^*$ for a commutative Krull monoid $\mathcal H^*$ (i.e., $\mathcal H^*$ is commutative, cancellative, completely integrally closed, and $v$-noetherian).

  \item There exists a weak transfer homomorphism $\theta' \colon H \to \mathcal B (G_0)$ for a subset $G_0$ of an abelian group.
  \end{itemize}
  If the second condition holds, then we say $H$ is a transfer Krull monoid over $G_0$.
  If $G_0$ is finite, then  $H$ is said to be a {\it transfer Krull monoid of finite type}.

  \item[(iii)] We say a domain $R$ is a {\it transfer Krull domain} (of finite type) if  its monoid of cancelative elements $R^{\bullet}$ is a transfer Krull monoid (of finite type).
  \end{enumerate}

  In particular,  commutative Krull monoids are  transfer Krull monoids.
   Rings of integers, holomorphy rings in algebraic function fields, and regular congruence monoids in these domains are commutative Krull monoids with finite class group such that every class contains a prime divisor (\cite[Section 2.11 and Examples 7.4.2]{Ge-HK06a}). Monoid domains and power series domains that are  Krull are discussed in \cite{Ki-Pa01, Ch11a}, and note that every class of a  Krull monoid domain contains a prime divisor. Thus all these commutative Krull monoids are transfer Krull monoids over a finite abelian group.

   However, a transfer Krull monoid need neither be commutative nor $v$-noetherian nor completely integrally closed. To give a noncommutative example, let $\mathcal O$ be a holomorphy ring in a global field $K$, $A$ a central simple algebra over $K$, and $H$ a classical maximal $\mathcal O$-order of $A$ such that every stably free left $R$-ideal is free. Then  $H$ is a transfer Krull monoid  over a ray class group of $\mathcal O$ (\cite[Theorem 1.1]{Sm13a}).
   Let $R$ be a bounded HNP (hereditary noetherian prime) ring. If every stably free left $R$-ideal is free, then its multiplicative monoid of cancelative elements  is a transfer Krull monoid (\cite{Sm16b}, Theorem 4.4).
    A class of  commutative weakly Krull domains which are transfer Krull but not  Krull  will be given in \cite[Theorem 5.8]{Ge-Sc-Zh17b}. Extended lists of commutative Krull monoids and of transfer Krull monoids, which are not commutative Krull, are given in \cite{Ge16c}.

\medskip

Let $G_0\subset G$ be a non-empty subset.  For a  sequence $S = g_1 \ldots g_{\ell} \in \mathcal F (G_0)$, we call
\[
\begin{aligned}
\mathsf k (S) & = \sum_{i=1}^l \frac{1}{\ord (g_i)} \ \in \mathbb Q_{\ge 0} \quad \text{the {\it cross number} of $S$, and } \\
\mathsf K (G_0) & = \max \{ \mathsf k (S) \mid S \in \mathcal A (G_0) \} \quad \text{the {\it cross number} of $G_0$}.
\end{aligned}
\]
 They were introduced by U. Krause in 1984 (see 
\cite{Kr-Za91a}) and  were studied under various aspects.
For the relevance with the theory of non-unique factorizations,  see \cite{Pl-Sc16a, Pl-Sc05b, Sc05d, Sc09c} and \cite[Chapter 6]{Ge-HK06a}.

Suppose $G\cong C_{q_1}\oplus\ldots\oplus C_{q_{r^*}}$, where $r^*$ is the total rank of $G$ and $q_1,\ldots, q_{r^*}$ are prime powers, and set
$$\mathsf K^*(G)=\frac{1}{\exp(G)}+\sum_{i=1}^{r^*}\frac{q_i-1}{q_i}\,.$$
It is easy to see that $\mathsf K^*(G)\le \mathsf K(G)$ and there is known no group for which  inequality holds. For further progress on $\mathsf K(G)$, we refer to \cite{Ga-Wa12a, He14a, Ki15a, Kr13a}.

\begin{lemma}\label{le2.1}
If $G$ is a $p$-group, then $\mathsf K(G)=\mathsf K^*(G)<\mathsf r(G)$.
\end{lemma}
\begin{proof}See \cite[Theorem 5.5.9]{Ge-HK06a}.
\end{proof}

A subset $G_0\subset G$ is called half-factorial if $\Delta (G_0) = \emptyset$. Otherwise, $G_0$ is called non-half-factorial. Furthermore, the set $G_0$ is called
\begin{itemize}

\item minimal non-half-factorial if it is non-half-factorial and every proper subset $G_1 \subsetneq G_0$ is half-factorial.

\item an LCN-set if $\mathsf k (A) \ge 1$ for all $A \in \mathcal A (G_0)$.
\end{itemize}

We collect some easy or well known results which will be used throughout the manuscript without further mention.
\begin{lemma}\label{2.3}
Let  $G_0\subset G$ be a non-empty subset. Then
\begin{enumerate}
\item $G_0$ is half-factorial if and only if $\mathsf k(A)=1$ for all $A\in \mathcal A(G_0)$.

\item If $G_0$ is an LCN-set, then $\min \Delta(G_0)\le |G_0|-2$.

\item  $\mathcal B(G_0)$ has accepted elasticity.

\item  $\rho(G)=\frac{\mathsf D(G)}{2}$.

\item $\Delta (G)$ is an interval  with $\min \Delta (G)=1$.

\item If $B\in \mathcal B(G)$, then $\rho(\mathsf L(B^k))\ge \rho(\mathsf L(B))$ for every $k\in \N$.

\item If $A\in\mathcal A(G)$, then $\{\exp(G), \exp(G)\mathsf k(A)\}\subset \mathsf L(A^{\exp(G)})$.


\end{enumerate}

\end{lemma}
\begin{proof}
1. and 2., see \cite[Proposition 6.7.3, Lemma 6.8.6]{Ge-HK06a}.
3. and 4., See \cite[Theorem 3.1.4,  Section 6.3]{Ge-HK06a}.
5. See \cite{Ge-Yu12b}.

6. Let $B\in \mathcal B(G)$ and $k\in \N$. Then $\max \mathsf L(B^k)\ge k\max\mathsf L(B)$ and $\min \mathsf L(B^k)\le k\min\mathsf L(B)$. It follows that $\rho(\mathsf L(B^k))=\frac{\max\mathsf L(B^k)}{\min\mathsf L(B^k)}\ge \frac{k\max\mathsf L(B)}{k\min\mathsf L(B)}=\rho(\mathsf L(B))$.

7. Let $A\in\mathcal A(G)$ and suppose $A=g_1\ldots g_{\ell}$, where $\ell\in \N$ and $g_1,\ldots, g_{\ell}\in G$. 
Then $$A^{\exp(G)}=\prod_{i=1}^{\ell}(g_i^{\ord(g_i)})^{\frac{\exp(G)}{\ord(g_i)}}\,.$$
Since $A$, $g_1^{\ord(g_1)},\ldots, g_{\ell}^{\ord(g_{\ell})}$ are atoms, we obtain  $\{\exp(G), \exp(G)\mathsf k(A)\}\subset \mathsf L(A^{\exp(G)})$.
\end{proof}

  We need the following lemma.

\begin{lemma}\label{2.1}
Let  $e_1,\ldots, e_r\in G$ be independent elements with the same order $n$, where $r,n\in \N_{\ge 2}$.  Then \begin{enumerate}
\item $$\Delta(\{e_1+\ldots+e_r, e_1,\ldots,e_r\})=\{r-1\}\,.$$
\item If $n\neq r+1 $, then $$\Delta(\{-(e_1+\ldots+e_r), e_1,\ldots,e_r\})=\{|n-r-1|\}\,.$$
\item If $n=r+1$, then $$\min\Delta(\{-(e_1+\ldots+e_r), e_1+\ldots+e_r, e_1,\ldots,e_r\})=r-1\,.$$
\end{enumerate}
\end{lemma}
\begin{proof}
1. See  \cite[Proposition  6.8.2]{Ge-HK06a}.

2. See \cite[Proposition 4.1.2.5]{Ge-HK06a}.

3. Let $g=e_1+\ldots+e_r$ and $G_0=\{g, -g, e_1, \ldots, e_r\}$.  Let $B\in \mathcal B(G_0)$ and assume
\[B=U_1\ldots U_k=V_1\ldots V_{\ell}, \text{ where } k,\ell\in N\text{ and } U_1,\ldots,U_k, V_1,\ldots, V_{\ell}\in \mathcal A(G_0)\,. \]
Let \begin{align*}
I_1=\{i\in [1,k]\mid \mathsf v_g(U_i)>0 \text{ and }\mathsf v_{-g}(U_i)=0\},& \ \ \  I_2=\{i\in [1,k]\mid U_i=g(-g)\}\\
J_1=\{j\in [1,\ell]\mid \mathsf v_g(V_j)>0 \text{ and }\mathsf v_{-g}(V_j)=0\},&\ \ \
J_2=\{j\in [1,\ell]\mid V_j=g(-g)\}\,.
\end{align*}
Then for every $i\in I_1$, $\mathsf k(U_i)=1+\frac{n-\mathsf v_g(U_i)}{n}(n-2)$ and for every $j\in J_1$, $\mathsf k(V_j)=1+\frac{n-\mathsf v_g(V_j)}{n}(n-2)$.

Note that for every $i\in[1,k]\setminus(I_1\cup I_2)$ and every $j\in [1,\ell]\setminus(J_1\cup J_2)$, $\mathsf k(U_i)=\mathsf k(V_j)=1$.
Therefore
$$\mathsf k(B)=k-|I_1|-|I_2|+\sum_{i\in I_1}(1+\frac{(n-\mathsf v_g(U_i))(n-2)}{n})+\sum_{i\in I_2}\frac{2}{n}=k+(n-2)|I_1|-\frac{n-2}{n}\mathsf v_g(B)$$ and $$\mathsf k(B)=\ell-|J_1|-|J_2|+\sum_{j\in J_1}(1+\frac{(n-\mathsf v_g(V_j))(n-2)}{n})+\sum_{j\in J_2}\frac{2}{n}=\ell+(n-2)|J_1|-\frac{n-2}{n}\mathsf v_g(B)\,.$$
It follows that $k-\ell=(n-2)(|J_1|-|I_1|)$ and hence $n-2\t \min \Delta(G_0)$. By 1., we obtain $n-2=r-1=\min \Delta(G_0)$.
\end{proof}

\medskip

If $d \in \mathbb N$ \ and \ $\ell, M \in \mathbb N_0$, then a finite subset $L\subset \Z$ \ is called an {almost arithmetical
 progression} ({\rm AAP} for short) with  difference $d$, length $\ell$, and
 bound $M$ if
 \begin{equation} \label{eq:defAAP}
 L = y + (L' \cup \{0, d , \ldots, \ell d\} \cup L'') \subset y + d \mathbb Z
 \end{equation}
 where $y\in \Z$,   $L' \subset [-M,-1]$, and $L'' \subset \ell d + [1, M]$.

\begin{lemma}\label{2.5}
There exist constants $M_1, M_2\in \N$ such that for every $A\in\mathcal B(G)$ with $\Delta(\supp(A))\neq \emptyset$, the set $\mathsf L(A^{M_1})$ is an AAP with difference $\min \Delta(\supp(A))$, length at least $1$,  and bound $M_2$.
\end{lemma}

\begin{proof}
See \cite[Theorem 4.3.6]{Ge-HK06a}.
\end{proof}

Next, we recall the definition of the invariants $\Delta^*(G)$ and $\Delta_1(G)$ (see \cite[Definition 4.3.12]{Ge-HK06a}) in the Characterization Problem.

Let
$$\Delta^*(G)=\{\min \Delta(G_0)\mid G_0\subset G \text{ is a subset with }\Delta(G_0)\neq \emptyset\}\,.$$
We define
\begin{align*}
\mathsf m (G) & = \max \{ \min \Delta (G_0) \mid G_0 \subset G \ \text{is an LCN-set with} \ \Delta (G_0) \ne \emptyset \} \,,
\end{align*}
and we denote by $\Delta_1 (G)$ the set of all $d \in \N$ with the following property:
\begin{itemize}
\item[]  For every $k \in \N$, there exists some $L \in \mathcal L (G)$  which is an AAP  with difference $d$ and length $\ell \ge k$.
\end{itemize}

\medskip
\begin{lemma} \label{le2}
Let $k\in \N$ be maximal such that  $G$ has a subgroup isomorphic to  $ C_{\exp(G)}^k$. Then

\begin{enumerate}
\item $\Delta^*(G)\subset \Delta_1(G)\subset \{d\in \Delta(G)\mid \text{there exists $d'\in \Delta^*(G)$ such that $d\t d'$}\}$.

\item $\max \Delta^*(G)=\max\Delta_1(G)=\max\{\exp(G)-2, \mathsf r(G)-1\}$.

\item  $\mathsf m(G) \le \max\{\mathsf r(G)-1, \left\lfloor\frac{\exp(G)}{2}\right\rfloor-1 \}$.


\item  $\{1, \mathsf r(G)-1,\exp(G)-2\}\subset\Delta^*(G)\subset \Delta_1(G)\subset [1, \max\{\mathsf r(G)-1, \lfloor\frac{\exp(G)}{2}\rfloor-1\}]\cup[\max\{1, \exp(G)-k-1, \exp(G)-2\}]$.

\item $\Delta_1(G)$ is an interval if and only if $\Delta^*(G)$ is an interval if and only if $\mathsf r(G)+k\ge \exp(G)-1$ or  $G\cong C_{2\mathsf r(G)+2}^{\mathsf r(G)}$.

\end{enumerate}

\end{lemma}
\begin{proof}1. follows from \cite[Corollary 4.3.16]{Ge-HK06a} and 2. from \cite[Theorem 1.1]{Ge-Zh15a}. 
3. See \cite[Proposition 3.7]{Zh17a}. 4. follows from 1. and  \cite[Theorem 1.1]{Zh17a}.

 5. If $\Delta_1(G)$ is an interval, then $\exp(G)-k-2\le \max\{\mathsf r(G)-1, \lfloor\frac{\exp(G)}{2}\rfloor-1\}$ by 4. which implies that $\Delta^*(G)$ is an interval by \cite[Theorem 1.1.2]{Zh17a}. If $\Delta^*(G)$ is an interval, then $\Delta_1(G)$ is an interval by 1.. It follows by \cite[Theorem 1.1.2]{Zh17a} that $\Delta^*(G)$ is an interval if and only if $\mathsf r(G)+k\ge \exp(G)-1$ or  $G\cong C_{2\mathsf r(G)+2}^{\mathsf r(G)}$.
\end{proof}

\medskip
\begin{lemma} \label{2.4}
Let $G_0\subset G$ with $\min \Delta(G_0)\ge \lfloor\exp(G)/2\rfloor$. Then
\begin{enumerate}

\item Let $A\in \mathcal A(G_0)$ with $\mathsf k(A)<1$.  Then  $\mathsf k(A)=\frac{\exp(G)-\min\Delta(G_0)}{\exp(G)}$ and $\ord(h)=\exp(G)$, $\mathsf v_h(A)=1$ for all $h\in \supp(A)$.

\item Let $A\in \mathcal A(G_0)$ with $\mathsf k(A)\ge 1$.  Then $\supp(A)$ is an LCN-set and $$ \min\Delta(\supp(A))\le |\supp(A)|-2\,.$$

\item If $G_0$ is a minimal non-half-factorial LCN-set with $\min \Delta(G_0)=\max\Delta^*(G)$,  then $|G_0|=\mathsf r(G)+1$ and for every $h\in G_0$, we have $\mathsf r(\langle G_0\setminus\{h\}\rangle)=\mathsf r(G)$.

\item Let $B\in \mathcal B(G)$ with $\rho(\mathsf L(B))=\rho(G)$. Suppose $B=U_1\ldots U_k=V_1\ldots V_{\ell}$, where $k=\min \mathsf L(B)$, $\ell=\max\mathsf L(B)$, and $U_1,\ldots ,U_k, V_1,\ldots V_{\ell}$ are atoms. Then $\mathsf k(U_i)\ge 1$ for all $i\in [1,k]$ and $\mathsf k(V_j)\le 1$ for all $j\in [1,\ell]$.
 \end{enumerate}
\end{lemma}
\begin{proof}

1.
 Since $\{\exp(G), \exp(G)\mathsf k(A)\}\subset \mathsf L(A^{\exp(G)})$, we have $$\min \Delta(\supp(A))\t \exp(G)-\exp(G)\mathsf k(A)$$ and hence $$\min \Delta(\supp(A))\le \exp(G)-\exp(G)\mathsf k(A)\le \exp(G)-2\,.$$
 Since $\min \Delta(G_0)\t \min \Delta(\supp(A))$ and $\min \Delta(G_0)\ge \lfloor\exp(G)/2\rfloor$, it follows that 
 $$
 \begin{aligned}
\frac{1}{2}(\exp(G)-2)&<\lfloor\exp(G)/2\rfloor\le \min\Delta(G_0)\le\min\Delta(\supp(A))\\
&\le\exp(G)-\exp(G)\mathsf k(A)\le \exp(G)-2\,.
 \end{aligned}$$ Therefore $\min\Delta(G_0)=\min\Delta(\supp(A))=\exp(G)-\exp(G)\mathsf k(A)$ and hence $\mathsf k(A)=\frac{\exp(G)-\min\Delta(G_0)}{\exp(G)}$.
 
Let $g\in \supp(A)$ such that $\frac{\ord(g)}{\mathsf v_g(A)}=\min \{\frac{\ord(h)}{\mathsf v_h(A)}\mid h\in \supp(A) \}$. Then $A^{\lceil\frac{\ord(g)}{\mathsf v_g(A)}\rceil}=g^{\ord(g)}\cdot B_1$, where $B_1\in \mathcal B(\supp(A))$. If there exists an atom $A_1$ with $\mathsf k(A_1)<1$ such that $A_1$  divides $B_1$, then  $\mathsf k(A_1)=\frac{\exp(G)-\min\Delta(G_0)}{\exp(G)}=\mathsf k(A)$. Therefore $1+\max\mathsf L(B_1)<\lceil\frac{\ord(g)}{\mathsf v_g(A)}\rceil$ and hence $\lfloor\exp(G)/2\rfloor\le\min \Delta(G_0)\le \lceil\frac{\ord(g)}{\mathsf v_g(A)}\rceil-2 $. It follows by the minimality of $\frac{\ord(g)}{\mathsf v_g(A)}$ that $\ord(h)=\exp(G)$ and $\mathsf v_h(A)=1$ for all $h\in \supp(A)$.

2. Assume to the contrary that $\supp(A)$ is not an LCN-set. Then there exists $A_1\in \mathcal A(\supp(A))$ with $\mathsf k(A_1)<1$. It follows by 1. that $\mathsf v_h(A_1)=1$ for all $h\in \supp(A_1)$ which implies that $A_1\t A$, a contradiction. Thus $\supp(A)$ is an LCN-set and hence $ \min\Delta(\supp(A))\le |\supp(A)|-2$.

3. By \cite[Lemma 4.2]{Ge-Zh15a}, we have $|G_0|=\mathsf r(G)+1$ and $\mathsf r(\langle G_0\rangle)=\mathsf r(G)$. If $h\in \langle G_0\setminus\{h\}\rangle$, then $\mathsf r(\langle G_0\setminus\{h\}\rangle)=\mathsf r(G)$. Otherwise let $d=\min\{k\in \N \mid kh\in \langle G_0\setminus h\rangle\}$ and hence $(G_0\setminus\{h\})\cup\{dh\}$ is also a minimal non-half-factorial LCN-set with $\min \Delta(G_0)=\max \Delta^*(G)$ by \cite[Lemma 6.7.10]{Ge-HK06a}. It follows that $\mathsf r(\langle G_0\setminus\{h\}\rangle)=\mathsf r(G)$.

4. Assume to the contrary that there exists $i\in [1,k]$, say $i=1$, such that $\mathsf k(U_1)<1$.
Then $\exp(G)\mathsf k(U_1)<\exp(G)$. Since $\exp(G)\mathsf k(U_1)\in \mathsf L(U_1^{\exp(G)})$, we infer  $$\min\mathsf L(B^{\exp(G)})\le \exp(G)\mathsf k(U_1)+\exp(G)(k-1)<\exp(G)k\,.$$ It follows by $\max\mathsf L(B^{\exp(G)})\ge \exp(G)\ell$  that $\rho(G)=\rho(\mathsf L(B))\ge\rho(\mathsf L(B^{\exp(G)}))> \frac{\ell}{k}=\rho(\mathsf L(B))$, a contradiction.

Assume to the contrary that there exists $j\in [1,\ell]$, say $j=1$, such that $\mathsf k(V_1)>1$.
Then $\exp(G)\mathsf k(U_1)>\exp(G)$. Since $\exp(G)\mathsf k(U_1)\in \mathsf L(U_1^{\exp(G)})$, we infer  $$\max\mathsf L(B^{\exp(G)})\ge \exp(G)\mathsf k(U_1)+\exp(G)(\ell-1)> \exp(G)\ell\,.$$  It follows by $\min\mathsf L(B^{\exp(G)})\le  \exp(G)k$ that  $\rho(G)=\rho(\mathsf L(B))\ge\rho(\mathsf L(B^{\exp(G)}))> \frac{\ell}{k}=\rho(\mathsf L(B))$, a contradiction.
\end{proof}

\bigskip
\section{The invariants $\rho(G, d)$, $\rho^*(G,d)$, and $\mathsf K(G,d)$}\label{3}
\bigskip

First, we introduce new invariants which  play a crucial role in the proof of Theorem \ref{main}
(see Proposition \ref{3.5}).

\begin{definition}\label{3.1}
Let  $d\in \Delta_1(G)$ and $k\in \N$. We define
\[\rho(G,d,k)=\sup\{\rho(L_k)\mid L_k\in \mathcal L(G) \text{ is an AAP with difference $d$ and length at least $ k$}\}\ge 1\,.
\]
Then $\big(\rho(G,d,k)\big)_{k=1}^{\infty}$ is a decreasing sequence of positive real numbers and hence converges. We denote by $\rho(G,d)$ the limit of $\big(\rho(G,d,k)\big)_{k=1}^{\infty}$.
\end{definition}

It follows by \cite[Theorem 3.5]{Ge-Zh17a} that $\rho(G, 1)=\rho(G)$ if and only if $G$ is not a cyclic group of order $4, 6$ or $10$.

\begin{lemma}\label{rho}
Let  $G_0\subset G$ be a subset with $\Delta(G_0)\neq \emptyset$. For every $B\in \mathcal B(G_0)$ with $\min \Delta(G_0)\in \Delta(\mathsf L(B))$, we have $\rho(G, \min \Delta(G_0))\ge \rho(\mathsf L(B))$.
\end{lemma}

\begin{proof}
Let $B\in \mathcal B(G_0)$ with $\min \Delta(G_0)\in \Delta(\mathsf L(B))$. By definition, $\mathsf L(B)$ is an AAP with difference $\min \Delta(G_0)$ and length at least $1$.  Therefore for every $k\in \N$,  $\mathsf L(B^k)$ is an AAP with difference $\min \Delta(G_0)$ and length at least $ k$. Thus  for every $k\in \N$, $\rho(G, \min \Delta(G_0), k)\ge \rho(\mathsf L(B^k))\ge \rho(\mathsf L(B))$ by Lemma \ref{2.3}.6. Therefore $\rho(G,\min \Delta(G_0))\ge \rho(\mathsf L(B))$.
\end{proof}

\begin{definition}\label{3.3}
Let  $d\in \Delta_1(G)$.  We define
\begin{align*}
 \rho^*(G,d)&=\max\{\rho(G_0)\mid G_0\subset G \text{ is a non-half-factorial subset  with $d$ divides }  \min\Delta(G_0)\}\,,\text{ and }\\
\mathsf K(G,d)&=\max\{\mathsf K(G_0)\mid G_0\subset G \text{ is a non-half-factorial subset with $d$ divides }\min\Delta(G_0)\}\,.
\end{align*}
\end{definition}

Note that $\rho^*(G, 1)=\rho(G)$ and $\mathsf K(G,1)=\mathsf K(G)$. For every $d\in \Delta_1(G)$, there always exists $G_0\subset G$ with $G_0$ non-half-factorial such that $d\t \min\Delta(G_0)$ by Lemma \ref{le2}.1. Since $\rho(G_0)>1$ and $\mathsf K(G_0)\ge 1$, we have
 $\rho^*(G, d)>1$ and $\mathsf K(G, d)\ge 1$. Furthermore, there exist $G_1, G_2\subset G$ with $d\t \min\Delta(G_1)$ and $d\t \min \Delta(G_2)$ such that $\rho^*(G,d)=\rho(G_1)$ and $\mathsf K(G,d)=\mathsf K(G_2)$.

\begin{lemma}\label{3.4}
Let $d\in \Delta_1(G)$. Then
\begin{enumerate}
\item  $\rho^*(G, d)\ge \rho(G,d)$.

\item $\rho^*(G,d)\le \rho(G, k_0d)$ for some $k_0\in \N$ with $k_0d\in \Delta_1(G)$.

\item $\rho^*(G,d)=\max\{\rho(G,kd)\mid k\in \N \text{ and }kd\in \Delta_1(G)\}$.
\end{enumerate}
\end{lemma}
\begin{proof} 
 1. By the definition of $\Delta_1(G)$ and $d\in  \Delta_1(G)$, 
for every $k\in\N$, we let $B_k\in \mathcal B(G)$ be such that $\mathsf L(B_k)$ is an AAP with difference $d$ and length at least $ k$. Let $$\ell_k=\min\{\min L_k\mid L_k\in \mathcal L(G) \text{ is an AAP with difference $d$ and length at least $ k$}\}$$ and hence  $\lim_{k\rightarrow \infty}\ell_k=\infty$ by $\rho(G)$ is finite.

By Lemma \ref{2.5}, there exists a constant
 $M$ such that $\min \Delta(\supp(B))\in \Delta(\mathsf L(B^M))$ for all $B\in \mathcal B(G)$ with $\Delta(\supp(B))\neq \emptyset$. Since $\lim_{k\rightarrow \infty}\ell_k=\infty$, we
 let $k\in \N$ be large enough such that $\min\mathsf L(B_k)\ge M|\mathcal A(G)|$
and  assume $$B_k=U_1^{t_1}\ldots U_{s}^{t_s}=V_1\ldots V_{\max\mathsf L(B_k)}\,,$$
 where $s, t_1,\ldots,t_s\in \N$, $t_1+\ldots+t_s=\min \mathsf L(B_k)$,  $U_1,\ldots, U_s$ are pair-wise distinct atoms,  and $V_1, \ldots, V_{\max\mathsf L(B_k)}$ are atoms.
 
 Since $\min\mathsf L(B_k)\ge M|\mathcal A(G)|$, we infer there exists $i\in [1,s]$, say $i=1$, such that $t_1\ge M$. Set $I=\{i\in [1,s]\mid t_i\ge M\}$ and $G_0=\supp(\prod_{i\in I}U_i)$. It follows by the choice of $M$ that 
 $$\min \Delta(G_0)\in \Delta (\mathsf L(\prod_{i\in I}U_i^{M}))\,. $$
 Since $\mathsf L(B_k)$ is an AAP with difference $d$ and $\prod_{i\in I}U_i^{M})$ is a subsequence of $B_k$,  we obtain $$d\t \gcd\Delta(\mathsf L(B_k))\t \gcd \Delta(\mathsf L(\prod_{i\in I}U_i^{M}))\,.$$
 Therefore $d\t \min \Delta(G_0)$  and hence $\rho(G_0)\le \rho^*(G, d)$. Since $|\prod_{i\not\in I }U_i^{t_i}|\le M|\mathcal A(G)|\mathsf D(G)$, there exists $J\subset [1, \max\mathsf L(B_k)]$ with $|J|\ge \max\mathsf L(B_k)-M|\mathcal A(G)|\mathsf D(G)$ such that $\prod_{j\in J}V_j$  divides $ \prod_{i\in I}U_i^{t_i}$. It follows by $\min\mathsf L(\prod_{i\in I}U_i^{t_i})=\sum_{i\in I}t_i$ that
\[
\begin{aligned}
\max\mathsf L(B_k)-M|\mathcal A(G)|\mathsf D(G)&\le |J|\le \max\mathsf L(\prod_{i\in I}U_i^{t_i})\\
&\le (\sum_{i\in I}t_i)\rho(G_0)\le \min \mathsf L(B_k)\rho^*(G,d)\,.
\end{aligned}\]
Therefore $$\rho(\mathsf L(B_k))\le \rho^*(G,d)+\frac{M|\mathcal A(G)|\mathsf D(G)}{\min \mathsf L(B_k)}\le \rho^*(G,d)+\frac{M|\mathcal A(G)|\mathsf D(G)}{\ell_k}\,.$$
By definition, we infer $\rho(G,d,k)\le \rho^*(G,d)+\frac{M|\mathcal A(G)|\mathsf D(G)}{\ell_k}$ which implies that $\rho(G,d)\le \rho^*(G,d)$.

2. 
Let $G_0\subset G$  with $d\t \min \Delta(G_0)$ and $\rho(G_0)=\rho^*(G,d)$.  Then there exists $B\in \mathcal B(G_0)$ such that $\rho(\mathsf L(B))=\rho^*(G,d)$. 
 Since $\supp(B)\subset G_0$, we infer $\min \Delta(G_0)$ divides $\min\Delta(\supp(B))$ and hence $d$ divides $ \min\Delta(\supp(B))$. 
 
 By Lemma \ref{2.5},  there exists a constant
  $M$ such that
$\min \Delta(\supp(B))\in \Delta(\mathsf L(B^M))$. It follows by Lemma \ref{rho} and Lemma \ref{2.3}.6 that $$\rho(G, \min \Delta(\supp(B)))\ge \rho(\mathsf L(B^M))\ge \rho(\mathsf L(B))=\rho^*(G,d)\,.$$ Thus the assertion  follows by  $d$ divides $ \min\Delta(\supp(B))$.

3. For every $k\in \N$ such that $kd\in \Delta_1(G)$, we have $\rho(G, kd)\le \rho^*(G, kd)\le \rho^*(G, d)$ by 1.. Therefore $\rho^*(G,d)\ge \max\{\rho(G,kd)\mid k\in \N \text{ and }kd\in \Delta_1(G)\}$. It follows by 2. that $\rho^*(G,d)=\max\{\rho(G,kd)\mid k\in \N \text{ and }kd\in \Delta_1(G)\}$.
\end{proof}

\begin{proposition}\label{3.5}
Suppose $\mathcal L(G)=\mathcal L(G')$ for some finite abelian group $G'$ and let $d\in \Delta_1(G)$. Then $d\in \Delta_1(G')$,  $\rho(G,d)=\rho(G', d)$, and $\rho^*(G, d)=\rho^*(G', d)$.
\end{proposition}

\begin{proof}
Since $\mathcal L(G)=\mathcal L(G')$, it follows by definition that $\Delta_1(G)=\Delta_1(G')$ and $\rho(G, kd)=\rho(G', kd)$ for every $k\in \N$ such that $kd\in \Delta_1(G)$. By Lemma \ref{3.4}.3, we obtain $\rho^*(G,d)=\rho^*(G',d)$. 
\end{proof}

\begin{lemma}\label{3.2}
Let $d\in \Delta_1(G)$. Then $\rho^*(G,d)\ge \mathsf K(G,d)$. In particular,
if $d\in [1,r-1]$, then $\rho^*(G,d)\ge \mathsf K(G,d)\ge 1+\frac{(n_1-1)d}{n_1}\ge 1+\frac{d}{2}$.
\end{lemma}

\begin{proof} 
 Suppose  $G_0\subset G$ with $d$ dividing $ \min \Delta(G_0)$ and $\mathsf K(G_0)=\mathsf K(G,d)$. Then there exists $A\in \mathcal A(G_0)$ such that $\mathsf k(A)=\mathsf K(G,d)\ge 1$.  Since $$\{\exp(G), \exp(G)\mathsf k(A)\}\subset \mathsf L(A^{\exp(G)})\,,$$ we have $$\rho^*(G,d)\ge \rho(G_0)\ge \rho(\mathsf L(A^{\exp(G)}))\ge \frac{\exp(G)\mathsf k(A)}{\exp(G)}=\mathsf k(A)=\mathsf K(G,d)\,.$$

In particular, if $d\in [1,r-1]$, we  let $e_1,\ldots, e_{d+1}$ be independent elements of order $n_1$. Set   $e_0=e_1+\ldots+e_{d+1}$ and $G_0=\{e_0,e_1,\ldots, e_{d+1}\}$. Then $\min \Delta(G_0)=d$ by Lemma \ref{2.1}.1.

 Since $A=e_0\cdot e_1^{n_1-1}\ldots e_{d+1}^{n_1-1}$  is an atom with $\mathsf k(A)=1+\frac{(n_1-1)d}{n_1}
$,  it follows that $$\mathsf K(G,d)\ge \mathsf K(G_0)\ge \mathsf k(A)\ge 1+\frac{(n_1-1)d}{n_1}\ge \frac{1+d}{2} \,.$$\qedhere
\end{proof}

\begin{lemma}\label{k}
Suppose $r\ge \lfloor\frac{n_r}{2}\rfloor+1$. Let  $G_0\subset G$ be a subset with $d\t \min\Delta(G_0)$, where $d\in \Delta_1(G)$ satisfies $r-1 \ge d\ge \lfloor\frac{n_r}{2}\rfloor$.  If  $A\in\mathcal A(G_0)$ with $\mathsf k(A)>1$,  then $\mathsf k(A)=1+\frac{sd}{n_{r-d}}$ for some $s\in \N$.

 In particular,
\begin{enumerate}

\item $\mathsf K(G, r-1)=1+\frac{s(r-1)}{n_1}$ for some  $s\ge n_1(u-\sum_{i=1}^u\frac{1}{p_i^{k_i}})$, where $n_1=p_1^{k_1}\ldots p_u^{k_u}$ with $u, k_1,\ldots, k_u\in\N$ and $p_1,\ldots, p_u$ are pairwise distinct primes.

\item if $n_r$ is a prime power, then $\mathsf K(G, r-1)=1+\frac{(n_1-1)(r-1)}{n_1}<r$.

\item if $n_1$ is not a prime power, then $\mathsf K(G, r-1)>r$.
\end{enumerate}
 \end{lemma}

\begin{proof} Let $t=r-d$ and we start with the following claim.

\medskip
\noindent{\bf Claim A: }
{\it Let $B\in \mathcal B(G_0)$ with  $\mathsf v_g(B)\equiv 0 \pmod {n_t}$ for each $g\in \supp(B)$ and $\supp(B)$ is an LCN-set. Then there exists $B_0\in \mathcal B(G_0)$ with $B_0$ is a product of atoms having cross number $1$ and $\mathsf v_g(B_0)\equiv 0 \pmod {n_t}$ for each $g\in \supp(B_0)$ such that $BB_0$ is a product of atoms having cross number $1$.}
\medskip

\noindent{\it Proof of {\bf Claim A}}.
Assume to the contrary that there exists a $B\in \mathcal B(G_0)$ with  $\mathsf v_g(B)\equiv 0 \pmod {n_t}$ for each $g\in \supp(B)$ and $\supp(B)$ is an LCN-set, 
such that the assertion does not hold.
Suppose $|\supp(B)|$ is minimal in all the counterexamples. 

Set $G_1=\supp(B)$. If for all $g\in G_1$,  $\ord(g)\mid \mathsf v_{g}(B)$, then $B$ is a product of atoms having cross number $1$, a contradiction. Therefore there exits $g_0\in G_1$ such that $\ord(g_0)\nmid \mathsf v_{g_0}(B)$.  Since $\mathsf v_g(B)\equiv 0 \pmod {n_t}$ for each $g\in \supp(B)$, we infer $\mathsf v_{g_0}(B)g_0\in \langle \{n_tg\mid g\in G_1\setminus\{g_0\}\}\rangle$ and $n_t\neq n_r$. 

 Let $\frac{n_r}{n_t}=q_1^{s_1}\ldots q_{v}^{s_{v}}$, where $v, s_1,\ldots, s_v\in \N$ and $q_1,\ldots, q_v$ are pairwise distinct primes.
 Let $i\in [1,v]$ and $H_i=\langle \{\frac{n_r}{n_tq_i^{s_i}}n_tg\mid g\in G_1\setminus\{g_0\}\}\rangle$. Then 
 $$\frac{n_r}{n_tq_i^{s_i}}\mathsf v_{g_0}(B)g_0\in H_i$$ and $H_i$ is an $q_i$-group of rank $\mathsf r(H_i)\le r-t=d$ which implies that there exists $E\subset G_1\setminus\{g_0\}$ with $|E|\le \mathsf r(H_i)\le d$ such that $$\frac{n_r}{n_tq_i^{s_i}}\mathsf v_{g_0}(B)g_0\in \langle \{\frac{n_r}{n_tq_i^{s_i}}n_tg\mid g\in E\}\rangle\,.$$
   Therefore
 there exists $B_i\in \mathcal B(G_1)$ such that $$|\supp(B_i)|\le d+1,\ \mathsf v_{g_0}(B_i)=\frac{n_r}{n_tq_i^{s_i}}\mathsf v_{g_0}(B),\ \text{ and } n_t\t \mathsf v_g(B_i) \text{ for every }g\in G_1\setminus\{g_0\}\,.$$ 
 Since $\supp(B_i)\subset G_1$ is an LCN-set, we have  $$\min \Delta(\supp(B_i))\le |\supp(B_i)|-2\le d-1\,.$$ Note that  $d\t \min \Delta(G_0)\t \min \Delta(\supp(B_i))$. We infer $\min\Delta(\supp(B_i))=0$ and hence $\supp(B_i)$ is half-factorial. Therefore  $B_i$ is a product of of atoms having cross number $1$. 
 
  Since $\gcd(\{\mathsf v_{g_0}(B_i)\mid i\in [1,v]\})=\mathsf v_{g_0}(B)$,  there exist $x_1,\ldots, x_{v}\in \N$ such that $$\mathsf v_{g_0}(BB_1^{x_1}\ldots B_{v}^{x_{v}})\equiv 0 \pmod {\ord(g_0)}\,.$$ Therefore $BB_1^{x_1}\ldots B_{v}^{x_{v}}=g_0^{y\ord(g_0)}C$ for some $y\in \N$ and $C\in \mathcal B(G_1\setminus\{g_0\})$. Note $n_t\t \mathsf v_g(B_i) $ for every $i\in [1,v]$ and every $g\in G_1\setminus\{g_0\}$. Thus
   $n_t\t \mathsf v_g(C)$ for each $g\in \supp(C)$.
   Since $|\supp(C)|<|\supp(B)|$, it follows  by the minimality of $|\supp(B)|$ that there exists $C_0\in \mathcal B(\supp(G_0))$ satisfying $C_0$ is a product of atoms having cross number $1$ and $n_t\t \mathsf v_g(C_0)$ for each $g\in \supp(C_0)$, such that $CC_0$ is a product of atoms having cross number $1$. Let $B_0=C_0B_1^{x_1}\ldots B_{t_1}^{x_{t_1}}$. Then $BB_0$ is a product of atoms having cross number $1$, a contradiction to our assumption.

\qed[Proof of {\bf Claim A}]

\medskip
Let $A\in \mathcal A(G_0)$ be with $\mathsf k(A)> 1$. Then Lemma \ref{2.4}.2 implies that $\supp(A)$ is an LCN-set. Set $B=A^{n_t}$ and hence {\bf Claim A} implies that there
 exist atoms $W_1,\ldots, W_{\ell}\in\mathcal A(G_0)$ having cross number $1$, where $\ell\in \N_0 $,   such that $A^{n_t}W_1\ldots W_{\ell}$ is a product of atoms having cross number $1$. Therefore $\{n_t\mathsf k(A)+\ell, n_t+\ell\}\subset \mathsf L(A^{n_t}W_1\ldots W_{\ell})$. Since $d\t \min \Delta(G_0)$,  we infer  $d\t (n_t\mathsf k(A)-n_t)$. It follows that $\mathsf k(A)=1+\frac{sd}{n_t}$ for some  $s\in \N$.

  \smallskip
 Now we begin to prove the "in-particular" parts.

1. For every $j\in [1,u]$ and every $m\in \Z$, we denote by $\lVert m\rVert_{j}$ the least positive residue of $m$ modulo $p_j^{k_j}$, that is, $\lVert m\rVert_{j}\in [1,p_j^{k_j}]$ and $\lVert m\rVert_{j}\equiv m \pmod {p_j^{k_j}}$.

 By definition of $\mathsf K(G, r-1)$, we have $\mathsf K(G, r-1)=1+\frac{sd}{n_1}$ for some  $s\in \N$.
Let $H$ be a subgroup of $G$ with $$H\cong C_{n_1}^r\cong C_{p_1^{k_1}}^{r-1}\oplus\ldots\oplus C_{p_u^{k_u}}^{r-1}\oplus C_{n_1}$$
 and let $(e_{1,1}, \ldots, e_{1,r-1} ,\ldots, e_{u,1}, \ldots, e_{u,r-1}, e)$ be a basis of $H$ with $\ord(e)=n_1$ and $\ord(e_{j,i})=p_j^{k_j}$ for all $i\in [1,r-1]$, $j\in [1,u]$. Set $e_0=e+\sum_{j=1}^u\sum_{i=1}^{r-1}e_{j,i}$ and $G_2=\{e_{1,1}, \ldots, e_{1,r-1} ,\ldots, e_{u,1}, \ldots, e_{u,r-1}, e, e_0\}$. For every $W\in \mathcal A(G_2)$, we have $$W=e_0^{\mathsf v_{e_0}(W)}e^{n_1-\mathsf v_{e_0}(W)}\prod_{i\in [1,r-1], j\in [1,u]}e_{j,i}^{p_j^{k_j}-\lVert \mathsf v_{e_0}(W)\rVert_{j}}$$ and
 $$ \mathsf k(W)=1+\sum_{j=1}^u\frac{p_j^{k_j}-\lVert \mathsf v_{e_0}(W)\rVert_{j}}{p_j^{k_j}}(r-1)\,.$$

We suppose that $U_1\ldots U_{\ell_1}=V_1\ldots V_{\ell_2}$ with $\ell_2-\ell_1=\min\Delta(G_2)$, where $\ell_1,\ell_2\in \N$ and $U_1,\ldots, U_{\ell_1}$, $V_1,\ldots, V_{\ell_2}\in \mathcal A(G_2)$.
Then $$\mathsf k(U_1\ldots U_{\ell_1})=\ell_1+\sum_{x=1}^{\ell_1}\sum_{j=1}^u\frac{p_j^{k_j}-\lVert \mathsf v_{e_0}(U_x)\rVert_{j}}{p_j^{k_j}}(r-1)=\ell_2+\sum_{y=1}^{\ell_2}\sum_{j=1}^u\frac{p_j^{k_j}-\lVert \mathsf v_{e_0}(V_y)\rVert_{j}}{p_j^{k_j}}(r-1)\,.$$
Since $\mathsf v_{e_0}(U_1\ldots U_{\ell_1})=\mathsf v_{e_0}(V_1\ldots V_{\ell_2})$, we infer $$\sum_{x=1}^{\ell_1}\lVert \mathsf v_{e_0}(U_x)\rVert_{j}\equiv\sum_{y=1}^{\ell_2}\lVert \mathsf v_{e_0}(V_y)\rVert_j\pmod{p_j^{k_j}}\text{ for each $j\in[1,u]$}\,.$$ 
Therefore 
$$
\begin{aligned}
&\sum_{x=1}^{\ell_1}\frac{p_j^{k_j}-\lVert \mathsf v_{e_0}(U_x)\rVert_{j}}{p_j^{k_j}}-\sum_{y=1}^{\ell_2}\frac{p_j^{k_j}-\lVert \mathsf v_{e_0}(V_y)\rVert_{j}}{p_j^{k_j}}\\
=&\ell_1-\ell_2-\frac{\sum_{x=1}^{\ell_1}\lVert \mathsf v_{e_0}(U_x)\rVert_{j}-\sum_{y=1}^{\ell_2}\lVert \mathsf v_{e_0}(V_y)\rVert_j}{p_j^{k_j}}\in \Z
\end{aligned}
$$ for each $j\in [1,u]$, whence $$\sum_{x=1}^{\ell_1}\sum_{j=1}^u\frac{p_j^{k_j}-\lVert \mathsf v_{e_0}(U_x)\rVert_{j}}{p_j^{k_j}}-\sum_{y=1}^{\ell_2}\sum_{j=1}^u\frac{p_j^{k_j}-\lVert \mathsf v_{e_0}(V_y)\rVert_{j}}{p_j^{k_j}}\in \Z\,.$$ 
Since  $\min\Delta(G_2)=\ell_2-\ell_1$, we infer  $r-1\t \min \Delta(G_2)$ which implies that $\mathsf K(G, r-1)\ge \mathsf K(G_2)$.
Let $W_1\in \mathcal A(G_2)$ be the atom with $\mathsf v_{e_0}(W_1)=1$. Then $\mathsf k(W_1)=1+(u-\sum_{i=1}^u\frac{1}{p_i^{k_i}})(r-1)$.

 Since
 $\mathsf K(G,r-1)=1+\frac{s(r-1)}{n_1}\ge \mathsf k(W_1)= 1+(u-\sum_{i=1}^t\frac{1}{p_i^{k_i}})(r-1)$ for some $s\in \N$, it follows that $s\ge n_1(u-\sum_{i=1}^u\frac{1}{p_i^{k_i}})$.

2. If $n_r$ is a prime power, then $\mathsf K(G)<r$ by Lemma \ref{le2.1}. It follows by  1. that
$r>\mathsf K(G, r-1)=1+\frac{s(r-1)}{n_1}\ge 1+\frac{(n_1-1)(r-1)}{n_1}$ for some $s\in \N$.  Therefore  $\mathsf K(G, r-1)=1+\frac{(n_1-1)(r-1)}{n_1}$.

3. If $n_1$ is not a prime power, then $u\ge 2$ and $u-\sum_{i=1}^u\frac{1}{p_i^{k_i}}>u-\frac{u}{2}\ge 1$. It follows by 1. that  $\mathsf K(G, r-1)>1+r-1=r$.
\end{proof}

\begin{proposition}\label{P-rho}
Let $s\in \N$ be maximal such that $G$ has a subgroup isomorphic to $ C_{n_r}^s$. Suppose $d\in \Delta_1(G)$ with $d\ge \max\{\left\lfloor\frac{n_r}{2}\right\rfloor, \left\lfloor\frac{r+1}{2}\right\rfloor\}$.  

\begin{enumerate}
\item  If $d\ge r$, then $\rho^*(G,d)=\frac{n_r}{n_r-d}$.

\item  If  $d\ge n_r-1$, then  $\rho^*(G,d)=\mathsf K(G,d)$.

\item Suppose $r=n_r-1\ge 3$. If $n_1=n_r=p^k$ for some prime $p$ and $k\in \N$, then $\mathsf K(G,r-1)<\rho^*(G, r-1)=1+\frac{(r+1)(r-1)}{r+2}<r$. Otherwise, $\mathsf K(G,r-1)=\rho^*(G, r-1)$.
\end{enumerate}
\end{proposition}

\begin{proof}
1. Since $d\ge \max\{\left\lfloor\frac{n_r}{2}\right\rfloor, \left\lfloor\frac{r+1}{2}\right\rfloor\}$ and $d\ge r$, it follows by  Lemma \ref{le2}(items 3. and 4.) that  $d>\mathsf m(G)$  and $d\in [n_r-s-1, n_r-2]$.

Let $e_1,\ldots, e_{n_r-d-1}$ be independent elements with order $n_r$ and let $e_0=-e_1-\ldots-e_{n_r-d-1}$.
Then $\min \Delta(\{e_0,e_1,\ldots,e_{n_r-d-1}\})=|n_r-(n_r-d-1)-1|=d$ by Lemma \ref{2.1}.2. Since $A=e_0e_1\ldots e_{n_r-d-1}$ is an atom, we infer $\mathsf L(A^{n_r})=\{n_r-d, n_r\}$. Therefore $$\rho^*(G, d)\ge\rho(\{e_0,e_1,\ldots,e_{n_r-d-1}\})\ge \rho(\mathsf L(A^{n_r}))=\frac{n_r}{n_r-d}\,.$$

Let $G_0\subset G$ with $d\t \min \Delta(G_0)$ be such that $\rho(G_0)=\rho^*(G,d)$. Then there exists $B\in \mathcal B(G_0)$ with $\rho(\mathsf L(B))=\rho^*(G,d)$. Set
$$B=U_1\ldots U_k=V_1\ldots V_{\ell},$$
where $k=\min \mathsf L(B)$, $\ell=\max\mathsf L(B)$, and $U_1,\ldots, U_k,V_1,\ldots, V_{\ell}\in \mathcal A(G_0)$ are atoms .

If there exists $i\in [1, k]$ such that $\mathsf k(U_i)>1$, then Lemma \ref{2.4}.2 implies that $\supp(U_i)$ is an LCN-set  and hence $\min\Delta(\supp(U_i))\le \mathsf m(G)$. Since  $$d\t \min \Delta(G_0)\t \min\Delta(\supp(U_i))\,,$$ we get a contradiction to $d>\mathsf m(G)$.
Therefore $\mathsf k(U_i)\le 1$ for all $i\in [1,k]$.

If there exists $i\in [1, \ell]$ such that $\mathsf k(V_i)<1$, then Lemma \ref{2.4}.1 implies that $\mathsf k(V_i)=\frac{n_r-d}{n_r}$.  Therefore $\mathsf k(V_i)\ge \frac{n_r-d}{n_r}$ for all $i\in [1,\ell]$.
It follows that $$k\ge \sum_{i=1}^k\mathsf k(U_i)=\mathsf k(B)=\sum_{i=1}^{\ell}\mathsf k(V_i)\ge \ell\frac{n_r-d}{n_r}\,.$$
 Then $\rho^*(G,d)=\rho(\mathsf L(B))=\frac{\ell}{k}\le\frac{n_r}{n_r-d}$ and hence $\rho^*(G,d)=\frac{n_r}{n_r-d}$.

2.
Let $G_0\subset G$ be such that $d\t\min \Delta(G_0)$ and $\rho(G_0)=\rho^*(G,d)$. If there exists an atom $A\in \mathcal A(G_0)$ such that $\mathsf k(A)<1$, then Lemma \ref{2.4}.1 implies that $\mathsf k(A)=\frac{n_r-d}{n_r}\le \frac{1}{n_r}$, a contradiction to $|A|\ge 2$. Thus $G_0$ is an LCN-set. Let $B\in \mathcal B(G_0)$ such that $\rho(\mathsf L(B))=\rho(G_0)=\rho^*(G,d)$. Then $$\min \mathsf L(B)\mathsf K(G,d)\ge \mathsf k(B)\ge \max\mathsf L(B)$$ which implies that $\rho^*(G,d)=\rho(\mathsf L(B))\le \mathsf K(G,d)$.

Let $G_0\subset G$ be such that $d\t\min \Delta(G_0)$ and $\mathsf K(G_0)=\mathsf K(G,d)$. Then there exists an atom $A\in \mathcal A(G_0)$ such that $\mathsf k(A)=\mathsf K(G,d)\ge 1$. Since $\{n_r, n_r\mathsf k(A)\}\subset \mathsf L(A^{n_r})$, we infer $$\rho^*(G,d)\ge\rho(G_0)\ge \rho(\mathsf L(A^{n_r}))\ge \mathsf k(A)=\mathsf K(G,d)$$ and hence $\rho^*(G,d)=\mathsf K(G,d)$.

3.
Let $r=n_r-1\ge 3$ and we proceed to prove the following claim.

\medskip
\noindent{\bf Claim B:}
{\it Suppose $\rho^*(G, r-1)>\mathsf K(G, r-1)$ and let  $G_0\subset G$  be  such that $(r-1)\t \min \Delta(G_0)$ and $\rho(G_0)>\mathsf K(G, r-1)$. 
\begin{enumerate}
\item[{\bf a.}] There exists $g\in G_0$ with $\ord(g)=n_r$ such that $-g\in G_0$.

\item[{\bf b.}] Let $G_2\subset G_0\setminus \{g,-g\}$ with $|G_2|=r$. If  there exists $a\in [1,n_r-1]$ such that $ag\in \langle G_2\rangle$, then $g=\sum_{h\in G_2}h$, $G\cong C_{n_r}^r$, and $G_2$ is a basis of $G$.

\item[{\bf c.}] $G\cong C_{n_r}^r$, $G_0=\{e_1,\ldots, e_r, g, -g\}$, where $g=e_1+\ldots+e_r$ and $(e_1, \ldots, e_r)$ is a basis of $G$,  and $\rho(G_0)=1+\frac{n_r(r-1)}{n_r+1}$. In particular, $\rho^*(G,r-1)=1+\frac{n_r(r-1)}{n_r+1}$.
\end{enumerate}
}

\smallskip
\noindent{\it Proof of {\bf Claim B.} } By Lemma \ref{le2}.2, we infer that $\min\Delta(G_0)=r-1=n_r-2=\max\Delta^*(G)$.

{\bf a.}  
If $G_0$ is an LCN-set, then for every $B\in \mathcal B(G_0)$, we have $$\min\mathsf L(B)\mathsf K(G, r-1)\ge \mathsf k(B)\ge \max\mathsf L(B)$$
 which implies that $\rho(\mathsf L(B))\le \mathsf K(G, r-1)$. Therefore  $\rho(G_0)\le \mathsf K(G, r-1)$, a contradiction. Thus there exists $A\in \mathcal A(G_0)$ such that $\mathsf k(A)<1$. Lemma \ref{2.4}.1 implies that $A=g(-g)$ for some $g\in G_0$ with $\ord(g)=n_r$. Hence $\{g,-g\}\subset G_0$.

\smallskip
{\bf b.} Let $E\subset G_2$ be minimal such that there exists $a\in [1, n_r-1]$ such that $ag\in \langle E\rangle$ and  let $d_g\in [1, n_r-1]$ be minimal such that  $d_gg\in \langle E\rangle$.
Then there exists an atom $V\in \mathcal A(E\cup\{g\})$ with $\mathsf v_g(A)=d_g$ and  $|\supp(V)|\le r+1$.
Let $V=g^{d_g}T$, where $T\in \mathcal F(E)$. Then $$V(-g)^{n_r}=(g(-g))^{d_g}((-g)^{n_r-d_g}T)\,, \text{ where } \mathsf L((-g)^{n_r-d_g}T)\subset [1, n_r-d_g]\,.$$
Note that for each $\ell\in \mathsf L((-g)^{n_r-d_g}T)$, we have $r-1\t d_g+\ell-2$. Therefore  $\mathsf L((-g)^{n_r-d_g}T)=\{n_r-d_g\}$ or ($d_g=1$ and $\mathsf L((-g)^{n_r-1}T)=\{1\}$). We distinguish two cases.

Suppose $\mathsf L((-g)^{n_r-d_g}T)=\{n_r-d_g\}$. Let $T=T_1\ldots T_{n_r-d_g}$ such that $(-g)T_i$, $i\in [1, n_r-d_g]$, are atoms, where $T_1, \ldots, T_{n_r-d_g}\in \mathcal F(E)$. Thus $-g\in \langle E\rangle$ which implies that $d_g=1$  by
the minimality of $d_g$. The minimality of $E$ implies that $\supp(T_i)=E$ for each $i\in [1, n_r-1]$. Then  for every $h\in E$, $\mathsf v_h(V)\ge n_r-1$. Therefore  $\ord(h)=n_r$ and  $\mathsf v_h(V)= n_r-1$. It follows that  $$T_1=\ldots =T_{n_r-1}=\prod_{h\in E}h\,.$$
If $\mathsf k((-g)T_1)<1$, then Lemma \ref{2.4}.1 implies that $T_1=g$, a contradiction. Therefore $$\mathsf k((-g)T_1)=\frac{1}{n_r}+\sum_{h\in E}\frac{1}{\ord(h)}=\frac{1+|E|}{n_r}\ge 1\,.$$ Since $|E|\le |G_2|\le r$ and $r=n_r-1$,
we have $|E|=|G_2|=r$. Let $G_2=\{e_1\ldots, e_r\}$. Then $V=ge_1^{n_r-1}\ldots e_r^{n_r-1}$ implies that $(e_1, \ldots , e_r)$ is a basis of $G$, $G\cong C_{n_r}^r$, and $g=e_1+\ldots+e_r$.

Suppose $d_g=1$ and $\mathsf L((-g)^{n_r-1}T)=\{1\}$. Note that $V=gT$ and hence $|T|\ge 2$. We infer 
 $\mathsf k((-g)^{n_r-1}T)>1$. It follows by Lemma \ref{2.4}.2 that $\{-g\}\cup E$ is an LCN-set and $\min \Delta(\{-g\}\cup E)\le |E|-1$. Since $(r-1)\t \min \Delta(\{-g\}\cup E)$ and $|E|\le |G_2| =r$, we have 
$$|E|=|G_2|=r \text{ and } \min \Delta(\{-g\}\cup G_2)=r-1\,.$$
Let $E_1\subset \{-g\}\cup G_2$ be a minimal non-half-factorial LCN-set. Then $$\min\Delta(E_1)=\min \Delta(\{-g\}\cup G_2)=r-1=\max\Delta^*(G)\,.$$
It follows by \cite[Lemma 4.2]{Ge-Zh15a} that $|E_1|=r+1$ and hence $E_1=\{-g\}\cup G_2$. Therefore 
\[\{-g\}\cup G_2 \text{ is a minimal non-half-factorial LCN-set.} \]
 Let $G_2=\{e_1, \ldots, e_r\}$ and assume 
 $$V=ge_1^{k_1}\ldots e_r^{k_r},\ \  V_1=(-g)^{n_r-1}e_1^{k_1}\ldots e_r^{k_r}\,,$$
  where $k_i\in [1, \ord(e_i)-1]$ for each $i\in [1,r]$.
  
 If $\mathsf k(V)>1$, then Lemma \ref{2.4}.2 and \cite[Lemma 4.5]{Ge-Zh15a} imply 
 \[\{g\}\cup G_2 \text{ is a  minimal non-half-factorial LCN-set with } \min \Delta(\{g\}\cup G_2)=\max\Delta^*(G)=r-1\,.\]
 By the minimality of $E=G_2$, we have for every $m\in [1, n_r-1]$ and every  $h\in G_2$, $mg\not\in \langle G_2\setminus\{h\}\rangle$. Note that $n_r\ge 4$.
Let $I=\{i\in [1,r]\mid k_i\ge \frac{\ord(e_i)}{2}\}$. Then \cite[Lemma 4.4.1]{Ge-Zh15a} implies that $$W_1=g^2\prod_{i\in I}e_i^{2k_i-\ord(e_i)}\prod_{i\in[1,r]\setminus I}e_i^{2k_i}$$ and $$W_2=(-g)^{n_r-2}\prod_{i\in I}e_i^{2k_i-\ord(e_i)}\prod_{i\in[1,r]\setminus I}e_i^{2k_i}$$ are both atoms. Since $V^2=\prod_{i\in I}e_i^{\ord(e_i)}W_1$ and $V_1^2=(-g)^{n_r}\prod_{i\in I}e_i^{\ord(e_i)}W_2$, we infer $r-1\t |I|-1$ and $r-1\t |I|$.  Hence $r=2$, a contradiction to $r\ge 3$.

 Thus $\mathsf k(V)=1$ which implies that $k_1=\ldots=k_r=1$ and $\ord(e_i)=n_r$ for each $i\in [1,r]$. It follows by $\mathsf k(V_1)>1$ and \cite[Lemma 4.3.2]{Ge-Zh15a} that $(-g)e_1^{n_r-1}\ldots e_r^{n_r-1}$ is also an atom. Therefore $e_1,\ldots, e_r$ are independent and  $g=e_1+\ldots+e_r$. Since $\mathsf r(G)=r$ and $\exp(G)=n_r$, we infer $(e_1,\ldots, e_r)$ is a basis of $G$ and $G\cong C_{n_r}^r$.

\medskip
{\bf c.} If for all $W\in \mathcal A(G_0)$, $\mathsf k(W)\le 1$, then for every $B\in \mathcal B(G_0)$, we have $$\min \mathsf L(B)\ge \mathsf k(B)\ge \max\mathsf L(B)\frac{2}{n_r}$$ which implies that $\rho(G_0)\le \frac{n_r}{2}$. It follows by Lemma \ref{3.2}  that $\mathsf K(G, r-1)\ge 1+\frac{r-1}{2}=\frac{n_r}{2}$, a contradiction to $\rho(G_0)>\mathsf K(G, r-1)$. 

Let $W\in \mathcal A(G_0)$ with $\mathsf k(W)> 1$. Then $\supp(W)$ is an LCN-set with $\min \Delta(\supp(W))=r-1$ by Lemma \ref{2.4}.2. Let $G_1\subset \supp(W)$ be a minimal non-half-factorial subset. Then $\min \Delta(G_1)=r-1$ and hence $|G_1|=r+1$ by \cite[Lemma 4.2.1]{Ge-Zh15a}.  Since $\{g,-g\}\not\subset \supp(W)$, we choose $h\in G_1$ such that $\{g,-g\}\cap (G_1\setminus\{h\})=\emptyset$.  Lemma \ref{2.4}.3 implies $\mathsf r(\langle G_1\setminus\{h\}\rangle)=r$.
Thus there exists $a\in [1, n_r-1]$ such that $ag\in \langle G_1\setminus\{h\}\rangle$. Since $|G_1\setminus\{h\}|=r$, 
it follows by {\bf a.} and {\bf b.} that $G\cong C_{n_r}^r$ and  there exists a basis $(e_1,\ldots, e_r)$ of $G$ such that $g=e_1+\ldots+e_r$ and $\{e_1,\ldots, e_r,g,-g\}\subset G_0$.

Assume to the contrary that there exists $h_0\in G_0\setminus \{e_1,\ldots, e_r,g,-g\}$. After renumbering if necessary, we may assume that $h_0=k_1e_1+\ldots+k_te_t$, where $t\in [1,r]$,  $k_i\in[1, n_r-1]$ for each $i\in[1,t]$ and $k_1=\min \{k_1,\ldots k_t\}$. Thus $k_1g\in \langle \{h_0, e_2,\ldots, e_r\}\rangle$. It follows by {\bf b.} that $g=h_0+e_2+\ldots+e_r$ and hence $h_0=e_1$, a contradiction. Therefore $G_0=\{e_1,\ldots, e_r, g, -g\}$.

We only need to prove $\rho(G_0)=1+\frac{n_r(r-1)}{n_r+1}$ which immediately implies that $\rho^*(G,r-1)=1+\frac{n_r(r-1)}{n_r+1}$.

Since $(ge_1^{n_r-1}\ldots e_r^{n_r-1})^{n_r}(-g)^{n_r}=(g(-g))^{n_r}(e_1^{n_r})^{n_r-1}\ldots (e_r^{n_r})^{n_r-1}$, we obtain $$\rho(G_0)\ge \frac{n_r+r(n_r-1)}{n_r+1}=1+\frac{n_r(r-1)}{n_r+1}\,.$$
Let $B\in \mathcal B(G_0)$ such that $\rho(\mathsf L(B))=\rho(G_0)$ and assume that
\[
B=U_1\ldots U_k=V_1\ldots V_{\ell}\,,\text{ and } \frac{\ell}{k}=\rho(G_0)\,,
\]where $k,\ell\in\N$ and $U_1,\ldots, U_k, V_1,\ldots, V_{\ell}$ are atoms.

Note that $\mathcal A(G_0)=\mathcal A(G_0\setminus\{-g\})\cup\{g(-g), (-g)e_1\ldots e_r, (-g)^{n_r}\}$. By Lemma \ref{2.4}.4, we have  $\mathsf k(U_i)\ge 1$ for all $i\in [1,k]$ and $\mathsf k(V_j)\le 1$ for all $j\in [1, \ell]$. Since $((-g)e_1\ldots e_r)^{n_r}=(-g)^{n_r}\prod_{i\in[1,r]}e_i^{n_r}$  and $\rho(B^{n_r})=\rho(G_0)$, substituting $B$ by $B^{n_r}$, if necessary, we can assume that $U_i\neq (-g)e_1\ldots e_r$, $V_j\neq (-g)e_1\ldots e_r$ for each $i\in [1,k]$ and each $j\in[1,\ell]$.

Since there must  exist $j_0\in [1,\ell]$ such that $V_{j_0}=g(-g)$, there must exists $i_0\in [1,k]$ such that
$U_{i_0}=(-g)^{n_r}$ which implies that $V_j\neq (-g)^{n_r}$ for all $j\in [1,\ell]$.
If  there exists $i_1\in [1,k]$ such that  $U_{i_1}= g^{n_r}$, then $\max\mathsf L(B)=n_r+\max\mathsf L(B(U_{i_0}U_{i_1})^{-1})$ which implies that
$$\rho(\mathsf L(B))=\frac{n_r+\max\mathsf L(B(U_{i_0}U_{i_1})^{-1})}{2+\min\mathsf L(B(U_{i_0}U_{i_1})^{-1})}=\rho(G)\,.$$ Therefore $$\rho(\mathsf L(B))=\frac{n_r}{2}=\rho(\mathsf L(B(U_{i_0}U_{i_1})^{-1}))\,.$$
It follows by Lemma \ref{3.2}  that $\mathsf K(G, r-1)\ge 1+\frac{r-1}{2}=\frac{n_r}{2}$, a contradiction to $\rho(G_0)>\mathsf K(G, r-1)$.
If  there exists $j_1\in [1,\ell]$ such that  $V_{j_1}= g^{n_r}$, then $\max\mathsf L(B(-g)^{n_r})=n_r+\ell-1$ which implies that
$$\rho(\mathsf L(B(-g)^{n_r}))\ge \frac{n_r+\ell-1}{k}>
\frac{\ell}{k}=\rho(G_0)\,,$$ a contradiction.

To sum up, we obtain $$\{U_i\mid i\in[1,k]\}\subset \mathcal (A(G_0\setminus\{-g\})\setminus\{g^{n_r}\})\cup\{(-g)^{n_r}\}$$ and $$\{V_j\mid j\in[1,\ell]\}\subset \mathcal A(G_0\setminus\{g, -g\})\cup\{(-g)g\}\,.$$
Let $I_1=\{i\in[1,k]\mid U_i=(-g)^{n_r}\} $, $I_2=\{i\in[1,k]\mid \mathsf v_g(U_i)\ge 1\} $, and  $J=\{j\in [1,\ell]\mid V_j=g(-g)\}$. 
Then
 $$n_r|I_1|=|J|=\mathsf v_{-g}(B)=\mathsf v_{g}(B)=\sum_{i\in I_2}\mathsf v_g(U_i)\ge |I_2|$$ and
\begin{align*}
\frac{2}{n_r}|J|+\ell-|J|=\mathsf k(B)=&\sum_{i\in I_2}\mathsf k(U_i)+k-|I_2|\\
            =&\sum_{i\in I_2}(1+\frac{n_r-\mathsf v_g(U_i)}{n_r}(n_r-2))+k-|I_2|\\
            =& k+(n_r-2)|I_2|-\frac{n_r-2}{n_r}\sum_{i\in I_2}\mathsf v_g(U_i)\\
            =& k+(n_r-2)|I_2|-\frac{|J|(n_r-2)}{n_r}\,.
\end{align*}
Therefore $k\ge |I_1|+|I_2|\ge \frac{n_r+1}{n_r}|I_2|$ and $\frac{\ell}{k}=1+\frac{|I_2|}{k}(n_r-2)\le 1+\frac{n_r(n_r-2)}{n_r+1}$. It follows that $\rho(G_0)=1+\frac{n_r(r-1)}{n_r+1}$.
\qed[Proof of {\bf Claim B.}]

\medskip
We distinguish two cases to finish the proof.

Suppose that $G\cong C_{p^k}^{p^k-1}$ for some  prime $p$ and $k\in \N$ with $p^k\ge 4$. Then $\mathsf K(G, p^k-2)=1+\frac{(p^k-1)(p^k-2)}{p^k}$ by Lemma \ref{k}.2.
Let $(e_1,\ldots , e_{p^k-1})$ be a basis of $G$ and $G_0=\{e_1,\ldots , e_{p^k-1}, g, -g\}$, where $g=\sum_{i\in[1,p^k-1]}e_i$. By Lemma \ref{2.1}.3, we have 
$p^k-2=\min \Delta(G_0)$. 
 Since $$(ge_1^{p^k-1}\ldots e_r^{p^k-1})^{p^k}(-g)^{p^k}=(g(-g))^{p^k}(e_1^{p^k})^{p^k-1}\ldots (e_r^{p^k})^{p^k-1}\,,$$ we obtain  $$\rho^*(G, p^k-2)\ge \rho(G_0)\ge \frac{p^k+(p^k-1)r}{p^k+1} =1+\frac{p^k(p^k-2)}{p^k+1}>\mathsf K(G, p^k-2)\,.$$
  It follows by {\bf Claim B.} that $$\mathsf K(G, p^k-2)<\rho^*(G, p^k-2)=1+\frac{p^k(p^k-2)}{p^k+1}<p^k-1\,.$$

Suppose that  $G\not\cong C_{p^k}^{p^k-1}$ for any  prime $p$ and any $k\in \N$. Assume to the contrary that $\mathsf K(G, r-1)<\rho^*(G, r-1)$. Then {\bf Claim B.} implies that $G\cong C_{n_r}^r$ and $\rho(G, r-1)<r$. Since $n_r$ is not prime power, it follows by Lemma \ref{k}.3  that $\mathsf K(G, r-1)>r$, a contradiction.
\end{proof}

\bigskip
\section{Proof of Main Theorems}\label{4}
\bigskip
 From now on, we assume $G'$ is a further finite abelian group with
 $G'\cong C_{n_1'}\oplus \ldots \oplus C_{n_{r'}'}$, where $r', n_1',\ldots, n_{r'}'\in \N$ and $1<n_1'\t \ldots \t n_{r'}'$. For convenience, we collect some necessary results which will be used all through the following two sections without further mention.

\begin{lemma}\label{4.1}
Suppose  $\mathcal L(G)=\mathcal L(G')$ and $d\in \Delta_1(G)$. 
Then \begin{enumerate}
\item If $G$ is isomorphic to a subgroup of $G'$, then $G\cong G'$.
\item $\max\{\mathsf r(G)-1,\exp(G)-2\}=\max\{\mathsf r(G')-1,\exp(G')-2\}$.
\item $d\in \Delta_1(G')$ and  $\rho^*(G,d)=\rho^*(G',d)$.
\end{enumerate}
\end{lemma}

\begin{proof}
1. By \cite[Proposition 7.3.1.3]{Ge-HK06a}, we have $\mathsf D(G)=\mathsf D(G')$.  It follows from \cite[Proposition 5.1.3.2 and 5.1.11.1]{Ge-HK06a} that $G\cong G'$.

2. follows from \cite[Corollary 4.3.16]{Ge-HK06a} and \cite[Theorem 1.1.3]{Ge-Zh15a}.

3. See Proposition \ref{3.5}. 
\end{proof}

\begin{theorem}\label{th1}
 Suppose $\mathcal L(G)=\mathcal L(G')$. Then
\begin{enumerate}
\item If $r\ge n_r-1$ and $n_1\neq 2$, then $\mathsf r(G)=\mathsf r(G')\ge \exp(G')-1$.

\item If $r\ge n_r-1$, $n_1\neq 2$, and $n_r$ is a prime power , then $\mathsf r(G)=\mathsf r(G')$ and $n_1=n_1'$.

\item If $r\le n_r-3$, then $\exp(G)=\exp(G')$ and $\mathsf r(G')\le n_r-3$.
\item If $\lfloor\frac{n_r}{2}\rfloor+1\le r\le n_r-3$, then $\exp(G)=\exp(G')$ and $\mathsf r(G)=\mathsf r(G')$.

\item If $ r\le n_r-3$ and $\Delta^*(G)$ is an interval, then $\exp(G)=\exp(G')$ and $\mathsf r(G)=\mathsf r(G')$.
\end{enumerate}
\end{theorem}

\begin{proof} 

1. Assume to the contrary that $\mathsf r(G')\neq\mathsf r(G)=r$. Since $r-1=\max\{\mathsf r(G)-1, \exp(G)-2\}=\max\{\mathsf r(G')-1, \exp(G')-2\}$, it follows that  $\mathsf r(G')-1<\exp(G')-2=r-1$.
Let $d=r-1$. Then $d\in \Delta_1(G)=\Delta_1(G')$. Since $\rho^*(G, d)\ge 1+\frac{(n_1-1)d}{n_1}$ by Lemma \ref{3.2} and $\rho^*(G', d)=\frac{\exp(G')}{2}$ by Proposition \ref{P-rho}.1,  it follows that $\rho^*(G', d)=\frac{\exp(G')}{2}=\rho^*(G, d)\ge 1+\frac{(n_1-1)(\exp(G')-2)}{n_1}$, a contradiction to $n_1\neq 2$. Thus $r-1=\mathsf r(G')-1=\max\{\mathsf r(G')-1, \exp(G')-2\}$ and hence $\mathsf r(G')\ge \exp(G')-1$.

\smallskip
2. Note that $r\ge n_r-1\ge 2$. If $r=2$, then $G\cong C_3\oplus C_3$. Therefore $\mathsf D(G)=\mathsf D(G')=5$ and $\max\{3-2, 2-1\}=\max\{\exp(G')-2, \mathsf r(G')-1\}$. It follows that $G\cong G'$. Thus we can assume $r\ge 3$.

By 1., we have $\mathsf r(G)=\mathsf r(G')\ge \exp(G')-1$. Since $n_r$ is a prime power and $n_r\ge 3$, it follows by Lemma \ref{k}.2 that $\mathsf K(G, r-1)=1+\frac{n_r(r-1)}{n_r}<r$.
Then $r\ge 3$ and Propositions \ref{P-rho}.2 and \ref{P-rho}.2 imply that $\rho^*(G, r-1)<r$ and hence $\rho^*(G', r-1)=\rho^*(G,r-1)<r$. Therefore  $$1+\frac{(n_1'-1)(r-1)}{ n_1'}\le \mathsf K(G', r-1)\le \rho^*(G', r-1)<r$$ by Lemma \ref{3.2}.  It follows by Lemma \ref{k}.1 that $\mathsf K(G', r-1)=1+\frac{(n_1'-1)(r-1)}{n_1'}$.

If $\mathsf K(G, r-1)=\rho^*(G, r-1)$ and  $\mathsf K(G', r-1)=\rho^*(G', r-1)$,  then $1+\frac{(n_r-1)(r-1)}{n_r}=1+\frac{(n_1'-1)(r-1)}{n_1'}$ which infers $n_1=n_r=n_1'$.

If $\mathsf K(G, r-1)<\rho^*(G, r-1)$ and  $\mathsf K(G', r-1)=\rho^*(G', r-1)$,  then $G\cong C_{n_r}^r$, $r=n_r-1$, and $\rho^*(G, r-1)=1+\frac{(n_r)(r-1)}{n_r+1}$ by Proposition \ref{P-rho}.3. Therefore $1+\frac{(n_r)(r-1)}{n_r+1}=1+\frac{(n_1'-1)(r-1)}{n_1'}$ which infers $r+2=n_r+1=n_1'\le \exp(G')$, a contradiction.

If $\mathsf K(G, r-1)=\rho^*(G, r-1)$ and  $\mathsf K(G', r-1)<\rho^*(G', r-1)$,  then $G'\cong C_{\exp(G')}^r$, $r=\exp(G')-1$,  and $\rho^*(G', r-1)=1+\frac{\exp(G')(r-1)}{\exp(G')+1}$ by Proposition \ref{P-rho}.3. Therefore
$1+\frac{(n_1-1)(r-1)}{n_1}=1+\frac{( \exp(G'))(r-1)}{ \exp(G')+1}$ which infers $n_r\ge n_1= \exp(G')+1=r+2$, a contradiction.

If $\mathsf K(G, r-1)<\rho^*(G, r-1)$ and  $\mathsf K(G', r-1)<\rho^*(G', r-1)$,  then $G\cong C_{r+1}^r\cong G'$ by Proposition \ref{P-rho}.3.

\smallskip
3. Note that $n_r\ge 4$.
Assume to the contrary that $\mathsf r(G')\ge n_r-2$. Let $d=n_r-3\in [1, \mathsf r(G')-1]$. Therefore  $d\in \Delta_1(G')=\Delta_1(G)$ and Lemma \ref{3.2} implies that $\rho^*(G',d)\ge 1+\frac{d}{2}$. Since $d\ge \max\{r, \lfloor n_r/2\rfloor\}$, Proposition \ref{P-rho}.1 implies that $\rho^*(G, d)=n_r/3< 1+\frac{d}{2}$, a contradiction to  $\rho^*(G,d)=\rho^*(G', d)$.

 Thus $\mathsf r(G')\le n_r-3$. Since $n_r-2=\max\{\mathsf r(G)-1, \exp(G)-2\}=\max\{\mathsf r(G')-1, \exp(G')-2\}$, it follows that  $n_r-2=\exp(G')-2$.

4. By 3, $\exp(G)=\exp(G')$  and  $\mathsf r(G')\le \exp(G')-3$. Assume to the contrary that $\mathsf r(G)\neq \mathsf r(G')$.

Suppose that $\mathsf r(G)>\mathsf r(G')$.
 Choose $d=r-1$.  Therefore $d\in \Delta_1(G)=\Delta_1(G')$ and Lemma \ref{3.2} implies that $\rho(G,d)\ge 1+\frac{r-1}{2}$. Since $d\ge \max\{\mathsf r(G'), \lfloor \exp(G')/2\rfloor\}$, Proposition \ref{P-rho}.1 implies that $\rho(G', d)=\frac{\exp(G')}{\exp(G')-d}\le \frac{n_r}{4}<1+\frac{r-1}{2}$, a contradiction to  $\rho^*(G,d)=\rho^*(G', d)$.

Suppose that $\mathsf r(G)<\mathsf r(G')$. Choose $d=r\in [1,\mathsf r(G')-1]$. Then $d\in \Delta_1(G')=\Delta_1(G)$ and Lemma  \ref{3.2} implies that $\rho^*(G',d)\ge 1+\frac{r}{2}$. Since $d\ge \max\{r, \lfloor n_r/2\rfloor\}$, Proposition \ref{P-rho}.1 implies that $\rho^*(G, d)=\frac{n_r}{n_r-r}=1+\frac{r}{n_r-r}<1+\frac{r}{2}$, a contradiction to  $\rho^*(G,d)=\rho^*(G', d)$.

5. By 3, we have $\mathsf r(G')+3\le \exp(G')=n_r$ and by 4,  we can assume that $r\le \lfloor\frac{n_r}{2}\rfloor$ and $\mathsf r(G')\le \lfloor\frac{n_r}{2}\rfloor$. Since $\Delta^*(G)$ is an interval, we obtain that $\Delta^*(G')=\Delta_1(G')=\Delta_1(G)=\Delta^*(G)$ is an interval by Lemma \ref{le2}.5. Let $k$ be maximal such that there exists a subgroup $H$ of $G$ with $H\cong C_{n_r}^k$ and let $k'$ be maximal such that there exists a subgroup $H'$ of $G'$ with $H'\cong C_{n_r}^{k'}$. Then $2r\ge r+k\ge n_r-2$ and $2\mathsf r(G')\ge \mathsf r(G')+k'\ge n_r-2$ by Lemma \ref{le2}.5.

Assume to the contrary that $r\neq \mathsf r(G')$ and by symmetry, we can assume $r< \mathsf r(G')$. Thus $\frac{n_r}{2}-1\le r< \mathsf r(G')\le \lfloor\frac{n_r}{2}\rfloor$ which implies that $n_r$ is even, $n_r=2r+2$, and $\mathsf r(G')=r+1$. Since $\Delta^*(G)$ and $\Delta^*(G')$ are  intervals, it follows  by Lemma \ref{le2}.5 that $G\cong C_{n_r}^r$ and $k'\ge r$. Then $G$ is a subgroup of $G'$ which implies that $G\cong G'$ by Lemma \ref{4.1}.1, a contradiction.
\end{proof}

\noindent{\bf Proof of Theorem \ref{main}.} By definition of transfer Krull monoids, it follows  that $\mathcal L(G)=\mathcal L(H)=\mathcal L(H')=\mathcal L(G')$.

 1.
Let $r\le n-3$. If $\Delta^*(G)$ is not an interval, then \cite[Theorem 1.1 and Theorem 1.2]{Zh17a} implies that $G\cong G'$.

Suppose $\Delta^*(G)$ is an interval. Then Theorem \ref{th1}.5 implies that $n=\exp(G')$ and $r=\mathsf r(G')$. Therefore $G'$ is isomorphic to a subgroup of $G$ which implies that $G\cong G'$ by Lemma \ref{4.1}.1.

2. If $n=2$, then $G$ is an elementary $2$-group and the assertion follows by \cite[Theorem 7.3.3]{Ge-HK06a}. We assume $n\ge 3$ is a prime power. Then Theorem \ref{th1}.2 implies that  $r=\mathsf r(G')$ and $n=n_1'$. Therefore  $G$ is isomorphic to  a subgroup of $G'$ which implies that  $G\cong G'$ by Lemma \ref{4.1}.1.
\qed

\section{Concluding remarks and conjectures}

\bigskip

Throughout this section, let $G\cong C_n^r$ with $\mathsf D(G)\ge 4$ and  $n=p_1^{k_1}\ldots p_u^{k_u}$, where $u,k_1,\ldots,k_u\in \N$ and $p_1,\ldots, p_u$ are pair-wise distinct primes.

\begin{conjecture}
Suppose $r\ge n-1$.
Then the following hold:

\begin{enumerate}
\item[\bf C1.] $\mathsf K(G, r-1)=1+\frac{s(r-1)}{n}$ for some $s\in \N$ with $\gcd(s,n)=1$.

\item[\bf C2.] $\mathsf K(G, r-1)=1+(\sum_{i=1}^u\frac{p_i^{k_i}-1}{p_i^{k_i}})(r-1)$.
\end{enumerate}

\end{conjecture}

It is easy to see that ${\bf C2}$ implies ${\bf C1}$. By Lemma \ref{k}.1, we have $\mathsf K(G, r-1)=1+\frac{s(r-1)}{n}$ for some $s\in \N$ and  if $n$ is a prime power, then ${\bf C2}$ holds.

\begin{proposition}\label{4.4}Let $G'$ be a  finite abelian group with $\mathcal L(G')=\mathcal L(G)$.
 If $r\ge n-1$ and  {\bf C1} holds, then $G\cong G'$.

\end{proposition}
\begin{proof}
If $n$ is a prime power, then the assertion follows by Theorem \ref{main}. 

Suppose $n$ is not a prime power. Then $r\ge n-1\ge 5$.
Let $\mathcal L(G)=\mathcal L(G')$ and let $G'\cong C_{n_1'}\oplus \ldots \oplus C_{n_{r'}'}$ with $r', n_1',\ldots, n_{r'}'\in \N$ and  $1<n_1'\t \ldots \t n_{r'}'$. 
Then Theorem \ref{th1}.1 implies that $r=r'\ge n_{r'}'-1$. By Lemma \ref{k}.1 and  Proposition \ref{P-rho}(items 2. and 3.), we have  $\rho^*(G,r-1)=\mathsf K(G, r-1)=1+\frac{s(r-1)}{n}$ for some $s\in \N$. If $\rho^*(G', r-1)=\mathsf K(G', r-1)$, then  Lemma \ref{k}.1 implies that $\rho^*(G', r-1)=1+\frac{s'(r-1)}{n_1'}$ for some $s'\in \N$. Thus $\rho^*(G, r-1)=\rho^*(G', r-1)$ implies that $\frac{s}{n}=\frac{s'}{n_1'}$ and hence $n\t n_1'$ by $\gcd(s,n)=1$. Therefore $G$ is isomorphic to a subgroup of $G'$ which implies that  $G\cong G'$ by Lemma \ref{4.1}.1.

Suppose $\rho^*(G', r-1)\neq \mathsf K(G', r-1)$. Then Lemma \ref{3.2} implies that $\rho^*(G', r-1)>\mathsf K(G', r-1)$ and hence Proposition \ref{P-rho}(items 2. and 3.) implies that $n_r'=r+1$ and $\rho^*(G, r-1)=1+\frac{(r+1)(r-1)}{r+2}$. Thus $\rho^*(G, r-1)=\rho^*(G', r-1)$ implies that $\frac{s}{n}=\frac{r+1}{r+2}$. Since $\gcd(s, n)=1$, we have $n=r+2$, a contradiction.
\end{proof}

Recall that $\mathsf K^*(G')\le \mathsf K(G')$ for all finite abelian group $G'$ and there is known no group $G'$ with $\mathsf K^*(G')<\mathsf K(G')$.

\begin{proposition}Let $G'$ be a  finite abelian group with $\mathcal L(G')=\mathcal L(G)$.
   If $r\ge \max\{(u-1)n+1, n\}\ge 3$ and  $\mathsf K(G)=\mathsf K^*(G)$, then  $G\cong G'$.
\end{proposition}

\begin{proof}
 If $u=1$, then the assertion follows by Theorem \ref{main}.2. Suppose $u\ge 2$.
 Note  $\mathsf K(G)=\mathsf K^*(G)=\frac{1}{n}+r(\sum_{i=1}^u\frac{p_i^{k_i}-1}{p_i^{k_i}})=1+\frac{s}{n}(r-1)+\sum_{i=1}^u\frac{p_i^{k_i}-1}{p_i^{k_i}}-\frac{n-1}{n}$, where $s=n(\sum_{i=1}^u\frac{p_i^{k_i}-1}{p_i^{k_i}})$. Since $r\ge (\mathsf \omega(n)-1)n+1$, we have
$\sum_{i=1}^u\frac{p_i^{k_i}-1}{p_i^{k_i}}-\frac{n-1}{n}< u-1\le \frac{r-1}{n}$. It follows by Lemma \ref{k}.1 that $1+\frac{s}{n}(r-1)\le \mathsf K(G, r-1)=1+\frac{s'}{n}(r-1)\le 1+\frac{s+1}{n}(r-1)$, where $s'\in \N$. Therefore $1+\frac{s}{n}(r-1)=\mathsf K(G, r-1)$. Since $s=n(\sum_{i=1}^u\frac{p_i^{k_i}-1}{p_i^{k_i}})$, we have $\gcd(s, n)=1$. The assertion follows by Proposition \ref{4.4}.
\end{proof}


\bibliographystyle{amsplain}

\end{document}